\documentclass[final,leqno,onefignum,onetabnum]{siamltex1213}

\usepackage{amssymb,amsmath,mathrsfs,bbm, changes}
\newtheorem{ass}{Assumption}
\newtheorem{rmk}{Remark}

\def\b{}
\def\s{\hat{s}}
\def\e{\varepsilon}
\def\sgn{\mbox{sgn}}
\def\minmod{\mbox{minmod}}

\newcommand{\esssup}{\text{esssup}}
\newcommand{\essinf}{\text{essinf}}

\title{A locally gradient-preserving reinitialization for level set functions}

\author{Lei Li\thanks{Department of Mathematics, University of Wisconsin-Madison, Madison, WI 53706 (Emails: \email{leili@math.wisc.edu}, \email{xxu@math.wisc.edu}, \email{spagnolie@math.wisc.edu}). The second author's research was supported by NSF-DMS grant 1159133.}\and Xiaoqian Xu\footnotemark[1]  \and Saverio E. Spagnolie\footnotemark[1] }

\begin{document}
\maketitle
\slugger{sisc}{xxxx}{xx}{x}{x--x}

\begin{abstract}
The level set method commonly requires a reinitialization of the level set function due to interface motion and deformation. We extend the traditional technique for reinitializing the level set function to a method that preserves the interface gradient. The gradient of the level set function represents the stretching of the interface, which is of critical importance in many physical applications. The proposed locally gradient-preserving reinitialization (LGPR) method involves the solution of three PDEs of Hamilton-Jacobi type in succession; first the signed distance function is found using a traditional reinitialization technique, then the interface gradient is extended into the domain by a transport equation, and finally the new level set function is achieved with the solution of a generalized reinitialization equation. We prove the well-posedness of the Hamilton-Jacobi equations, with possibly discontinuous Hamiltonians, and propose numerical schemes for their solutions. A subcell resolution technique is used in the numerical solution of the transport equation to extend data away from the interface directly with high accuracy. The reinitialization technique is computationally inexpensive if the PDEs are solved only in a small band surrounding the interface. As an important application, the LGPR method will enable the application of the local level set approach to the Eulerian Immersed boundary method. 
\end{abstract}

\begin{keywords}
Level set function; Reinitialization; Interface gradient; Eulerian Immersed Boundary Method; Discontinuous Hamiltonian
\end{keywords}

\section{Introduction}
The level set method \cite{os88, of01} is a classical framework used to accurately and elegantly evolve Lagrangian interfaces over a fixed Eulerian grid. It has seen very wide application in numerous fields, from fluid-structure interactions (e.g., lipid vesicles \cite{sm11}, bubbles \cite{chmo96}, two-phase flows \cite{sso94}) to image processing \cite{ms95}, computational geometry \cite{sethian1999level}, computer vision \cite{sethian1999level}, and materials science \cite{lzg99,sethian1999level}. The level set method involves the tracking of a level set function $\phi$, a continuous function with the property that its zero level set $\Gamma=\{x: \phi(x)=0\}$ represents the Lagrangian interface (for instance, the boundary between two fluid phases or an immersed elastic structure). However, if the interface is deformed by a velocity field, for instance, then the gradient of the associated level set function, $\nabla\phi$, may grow unbounded in the process. To reduce the associated numerical error the level set function is commonly reinitialized. Even if the boundary is not highly deformed, when a local level set method \cite{as95,pmozk99} is applied to reduce computational costs, reinitialization is required if the interface encroaches the boundary of the thin computational tube.

For many applications, only the position and curvature of the interface are needed, and the level set function $\phi$ after each reinitialization may be chosen to be a signed distance function \cite{sso94, chmo96, ms95}. For example, in the simulation of elastic structures immersed in a fluid, if the tension is assumed constant (see \cite{chmo96}) then the force depends only on the curvature of the interface so that the signed distance function contains sufficient information. However, in the Eulerian immersed boundary method \cite{cm04,cm06}, $|\nabla\phi|_{\Gamma}$ represents the stretching of the elastic structure. Consequently, the elastic forces depend on $|\nabla\phi|$ at the interface and the signed distance function cannot be used to compute these forces. One solution to this problem, shown by Cottet \& Maitre, is to avoid reinitialization altogether and to instead to renormalize with a particular approximation of the Dirac delta function used in interface capture \cite{cm06}. However, there are situations in which this strategy is inadequate. For instance, it would not be effective in the local level set approach to the Eulerian immersed boundary method.

In this paper we develop a method for reinitializing the level set function that locally preserves its gradient near the Lagrangian interface. The proposed locally gradient-preserving reinitialization (LGPR) method involves the solution of three Hamilton-Jacobi equations in succession; first the signed distance function is found using the traditional reinitialization technique, then the cost function is obtained by extending the interface gradient into the domain by a transport equation, and finally the new level set function is achieved by the solution of a generalized reinitialization equation with the cost function obtained in the previous step. The steady reinitialization equation is an Eikonal equation with the cost function discontinuous at the cut locus of the interface. We show that the "proper" viscosity solution (to be defined) of the Eikonal equation exists and is unique. We also prove that the viscosity solution that vanishes at the interface of the reinitialization equation converges to this proper viscosity solution and hence it is the level set function desired. We then propose numerical schemes for their fast and accurate solution. As an important application, the LGPR method will enable the application of the local level set approach to the Eulerian Immersed boundary method, which may then be comparable in cost with the classical immersed boundary method of Peskin \cite{peskin02}, but with improved stability. 

LGPR consists of some equations that are very similar with some well-studied equations in literature. Motivated by those results, some theoretic results are new in this paper.  For example, about the Eikonal equation with boundary conditions imposed on the outer boundary, some results are available for some discontinuous cost functions. Due to the assumptions imposed in those references, the results can't be applied to our case for the proof of the uniqueness. We provide a new proof for the uniqueness for our special Eikonal equation. \deleted{That no cusps with zero angles in the cut locus is interesting, though we don't use this result to show the uniqueness.} The formula for the solution of the general reinitialization equation, as far as we know, is new, though the generalization from existing results is not hard.  The numerical schemes are combinations of modified versions of some well-known methods except that we propose a new upwind scheme for the transport equation for extending quantities out from the interface. 

The paper is organized as follows. In \S\ref{sec:setup} we present the sequence of PDEs involved in locally gradient-preserving reinitialization. In \S\ref{sec:theory} we show the theoretical results, and give explicit formulas for viscosity solutions. Numerical schemes for solving the equations are the topic of \S\ref{sec:numerics}, and a few illustrative examples are provided in \S\ref{sec:exp}. We conclude with a brief summary in \S\ref{sec:conc}. Proofs for several claims made throughout the paper about the cut locus, existence and uniqueness of the proper viscosity solution of the Eikonal equation with a discontinuous cost function, and other issues are included in the appendix.
 
\section{Problem setup and reinitialization method}\label{sec:setup}
We begin by describing in more detail the motivation and setup of the problem, and presenting the locally gradient-preserving reinitialization method. For the sake of presentation, we will consider as a model problem a closed one-dimensional elastic interface embedded in $\mathbb{R}^2$, though the method could be extended into cases with several closed interfaces or higher dimensions without conceptual difficulty.

Suppose $\phi$ is a level set function such that the zero level set $\Gamma$ agrees with the interface $X(\xi, t)$, where $\xi$ is a Lagrange coordinate and $t$ is time. Assume that $\phi>0$ inside $\Gamma$ and $\phi<0$ outside $\Gamma$. In \cite{cm04}, it was shown that $|\nabla\phi(X(\xi, t), t)|/|X_{\xi}(\xi, t)|=\alpha(\xi)$ is independent of $t$ when $\phi$ is convected by the velocity field, and thus if $\phi$ is constructed initially such that $\alpha=1$, $|\nabla\phi|_{\Gamma}$ measures the tangential stretching (or compression) of the interface. Generically, such an elastic structure responds energetically to both bending and stretching deformations. The elastic force due to interface bending depends on the curvature $\kappa=-\nabla\cdot \b{\hat{n}}$, where $\b{\hat{n}}=\nabla\phi/|\nabla\phi|$ is the inward-pointing normal vector at the interface, which is unchanged under any reinitialization scheme that preserves the location of the level set. The elastic force at a point $x$ due to interface stretching, however, is given by (\cite{cm04}):
\begin{multline}\label{eq:F_elastic}
\b{F}(x)=\nabla\left(E'(|\nabla\phi|)\right)|\nabla\phi|\delta(\phi)-\nabla\cdot\left(E'(|\nabla\phi|)\frac{\nabla\phi}{|\nabla\phi|}\right)\nabla\phi \,\delta(\phi)\\
=\kappa\,E'(|\nabla\phi|) \nabla\phi \,\delta(\phi)+E''(|\nabla\phi|)\b{\hat{n}}\cdot\nabla \nabla\phi\cdot(\b{I-\hat{n}\hat{n}})|\nabla\phi|\delta(\phi)
\end{multline}
where $\b{I}$ is the identity operator, $\hat{n}\hat{n}$ is a dyadic product and $E(\cdot)$ is the elastic energy due to stretching. The first term in \eqref{eq:F_elastic} is a force due to a curved interface under a certain tension, while the second term is due to tension gradients along the interface. We introduce the stretch function
\begin{gather}
\chi(x)=|\nabla\phi|(x), \ \  x\in\Gamma
\end{gather}
defined on the interface (time dependence is ignored). Stretching occurs in regions where $\chi>1$, and compression occurs where $\chi<1$. In the above, we require two quantities that may be tied to the gradient of the level set function: $\chi(x)$ and $\b{\hat{n}}\cdot\nabla\nabla\phi\cdot(\b{I-\hat{n}\hat{n}})|_{\Gamma}=D_s \chi(\Gamma(s))\s$, where $s$ is the arc length parameter, $D_s=d/ds$ and $\s=\Gamma'(s)$ is the unit tangent vector along the surface $\Gamma$.

In the process of the convection of the interface, $\nabla\phi$ may have become unbounded (usually away from the interface), or the zero level set may have drifted towards the boundary of a tube in the local level set method. In this situation it is necessary to find a new level set function that is better behaved. So as to leave the elastic force unchanged during this process, the stretch function $\chi$ must be preserved during reinitialization. In theory preserving $\chi(\Gamma(s))$ is sufficient, but in numerical application we must also ensure that its tangential derivative is accurately preserved. We now formulate the reinitialization problem in a more mathematical way. 

\subsection{Locally gradient-preserving reinitialization}
Suppose that $\phi_0$ is a uniformly continuous level set function, $C^1$ on $\Gamma=\{x:\phi_0(x)=0\}$ with $x\in \mathbb{R}^2$, but not necessarily $C^1$ elsewhere. $\phi_0$ is assumed to be positive inside the interface $\Gamma$ and negative outside $\Gamma$. In addition we assume that $\Gamma$ satisfies:
\begin{ass}\label{ass:gamma}
$\Gamma$ is a closed, nonintersecting $C^1$ curve which can be decomposed into several segments, each of which is locally analytical throughout (including at the segment endpoints). 
\end{ass}

Consider an arc-length parameterization of the interface on one such segment, $\Gamma(s):[a,b]\to\mathbb{R}^2$. $\Gamma$ is locally analytical if for every $s_0\in[a,b]$, there is a number $\e>0$, so that the Taylor series of $\Gamma$ about $s_0$ converges to $\Gamma$ in $(s_0-\e, s_0+\e)\cap [a,b]$. That the segment endpoints are also assumed to be analytical (one-sided) removes certain pathological behaviors \cite{ccm97}. The assumption on $\Gamma$ makes physical sense for practical interfaces.

We denote by $U$ the open domain enclosed by $\Gamma$. The stretch function
\begin{gather}
\chi(x)=|\nabla\phi_0|(x), \ \  x\in\Gamma
\end{gather}
is assumed to satisfy: $\chi(\Gamma(s))$ is continuous and the derivative $D_s\chi(\Gamma(s))$ is piecewise continuous. We assume $0<c_1\leq \chi(\Gamma(s))\leq c_2$ for two constants $c_1,c_2$, which is a physically relevant constraint since the stretching deformation is generally bounded when the material is elastic.  We aim to find a new level set function $\phi$ which is Lipschitz continuous (the gradient is bounded), smooth $(C^1)$ in a local band around $\Gamma$, and in particular, preserves the interface gradient, $|\nabla\phi|(x\in\Gamma)=\chi(x\in\Gamma)$. 

We are then led to the Eikonal equation,
\begin{gather}\label{eq:eikonal}
\begin{array}{c}H(x,\nabla \phi)=\sgn(\phi_0)(|\nabla \phi|-f(x))=0 \ \ \ x\in \mathbb{R}^2,\\
\phi(x)=0,\ x\in \Gamma,\end{array}
\end{gather}
for some suitable $f$ that has the boundary condition $f(x)=\chi(x)$ for $x\in\Gamma$. Here $H(x,p)=\sgn(\phi_0(x)) (|p|-f(x))$ is the Hamiltonian. The sign function $\sgn(\phi_0)$ connects the level set function to the so-called viscosity solution, as will be discussed in the next section. While $f$ is known on $\Gamma$, part of the reinitialization process will be first to extend $f$ away from the interface and into the larger domain. 

In the traditional reinitialization procedure, the new level set function $\varphi$ is the signed distance function which is recovered by solving numerically a Hamilton-Jacobi (H-J) equation \cite{sso94},
\begin{gather}\label{eq:dis}
   \begin{array}{c} \displaystyle \frac{\partial\varphi}{\partial\tau}+\sgn(\phi_0)(|\nabla\varphi|-1)=0,\\
    \ \ \ \varphi(x,0)=\phi_0(x),\end{array}
\end{gather}
which is inadequate in our effort to preserve the interface stretch information.

Instead, we propose continuing the process by two extra steps to find a new function $\phi$ that shares its gradient with $\phi_0$ locally near $\Gamma$. First, we extend $f(x\in\Gamma)=\chi(x\in\Gamma)$ from the interface out into the whole domain along the characteristic lines of the signed distance function by a transport equation,
\begin{gather}\label{eq:extension}
\begin{array}{c}   \displaystyle\frac{\partial f}{\partial \tau}+\sgn(\varphi) \displaystyle\nabla\varphi\cdot\nabla f=0,\\
    f(x\in\Gamma,\tau)=\chi(x). \end{array}
\end{gather}
The desired level set function is then obtained by solving a generalized reinitialization equation,
\begin{gather}\label{eq:levelfunc}
\begin{array}{c} \displaystyle\frac{\partial\phi}{\partial\tau}+\sgn(\phi_0)(|\nabla\phi|-f(x))=0,\\
 \phi(x,0)=\phi_0(x),\end{array}
\end{gather}
For the remainder of the paper, references to the ``reinitialization equation'' are to \eqref{eq:levelfunc}; the earlier equation, \eqref{eq:dis}, a special case of \eqref{eq:levelfunc}, will be referred to as the traditional reinitialization equation. For convenience we have abused the notation for $f(x,\tau)$ and the steady cost function $f(x)$, and similarly $\phi(x,\tau)$ and $\phi$. Whether we mean the steady solution or the pseudo-time dependent solution should be clear by the context. We refer to \eqref{eq:dis}-\eqref{eq:levelfunc} as the locally gradient-preserving reinitialization (LGPR) method.

The LGPR method proposed above is straight-forward and there are no immediately apparent complications, but it is not obvious that the solution of \eqref{eq:levelfunc} converges to the solution of the Eikonal equation, \eqref{eq:eikonal}, or even if it exists since the cost function $f$ developed with the transport equation in \eqref{eq:extension} may be discontinuous.  However, as we will show in the following section, the solutions so obtained are well-defined so that the reinitialization method presented here may become a basis for fast, accurate local level set methods.  We will also numerically determine how to preserve the interface gradient $\chi$ so that $D_s\chi(\Gamma(s))$ can be recovered accurately. \\
\begin{rmk}
When $f(x)$ is known in the whole domain, the Eikonal equation, Eq~\eqref{eq:eikonal} can be solved using, for instance, the Fast Marching Method (FMM) \cite{sethian96} or the Fast Sweeping Method (FSM) \cite{zhao04}. This is one approach for initializing an original level set function, and could also be used as an alternative basis for reinitialization.
\end{rmk}\\

\section{Theoretical results}\label{sec:theory}
In this section we will show that the LGPR equations described in the previous section are well-posed. Specifically, we will show that the cost function $f(x)$ is continuous outside of a closed set consisting of arcs and vertices and that the Eikonal equation has a unique ``proper'' solution (to be clarified later) given the function $f(x)$ produced using the transport equation, \eqref{eq:extension}. A formula for $\phi(x,\tau)$  is derived,  which is found to converge to the ``proper'' solution of the Eikonal equation \eqref{eq:eikonal} in finite time. Thus, \eqref{eq:levelfunc} is shown to be equivalent to \eqref{eq:eikonal}. 

Note that there are some results about the Eikonal equation or some Hamilton-Jacobi equations with discontinuous Hamiltonians, which can't applied to our case, as we will see later. For the Eikonal equation, one can find some results in  \cite{ostrov00,soravia02, soravia06, de04, ff14}. In these references, one can check that their assumptions don't apply to the cost function generated by \eqref{eq:extension}. In \cite{ai02}, the equation $u_t+\sgn(u_0)H(\nabla u)=0$ has been studied but \eqref{eq:levelfunc} doesn't belong to this class. In our theoretical  results, the proof for the uniqueness of the Eikonal equation is new and the formula for $\phi(x,\tau)$ as far as we know, is derived for the first time, though it is not hard to obtain this formula from existing results.

\subsection{The Eikonal equation}
The Eikonal equation, \eqref{eq:eikonal}, is a Hamilton-Jacobi equation with Hamiltonian $H(x,p)=\sgn(\phi_0(x)) (|p|-f(x))$, and $f$ is called the cost function. We first introduce the definition of viscosity solutions (see \cite{cl83, Ishii85}), then we study the continuity (and regions of discontinuity) of $f$ from its development by \eqref{eq:extension}, and finally explore the associated solutions of the Eikonal equation.

\subsubsection{Viscosity solutions of the Eikonal equation}\label{sec:Viscosity solutions of the Eikonal equation} In the general setting of the Eikonal equation, solutions need not exist in the classical sense. Instead, solutions are developed in a weaker sense; specifically, a viscosity solution is defined as follows. \\

\begin{definition}\label{def:vs1}
A viscosity sub-solution (super-solution) of $H(x,\nabla \phi)=0$ is an upper semi-continuous function (a lower semi-continuous function), if for any $C^{\infty}$ function $\zeta$, when $\phi-\zeta$ has a local maximum (minimum) at $x_0$ which is an interior point, then $H_*(x,\nabla\zeta(x_0))\leq0$ $(H^*(x,\nabla\zeta(x_0))\geq0)$. A viscosity solution is a continuous function that is both a sub- and super-solution.
\end{definition}\\

In this definition, $H^*(x, p)=\lim_{r\to 0}\esssup \{H(y, q)|\ \|(y,q)-(x,p)\|_{L^2}\leq r \}$ is the sup-envelope, and $H_*$ is similarly defined to be the inf-envelope. 

For example, consider $|u'|=1, x\in(-1,1)$ with $\Gamma=\{-1,1\}$ (see exercises in \cite{evans10}) whose viscosity solution is $u=1-|x|$ for $x\in[-1,1]$. The viscosity solution of $-|u'|=-1, x\in(-1,1)$ with $\Gamma=\{-1,1\}$ is $u=|x|-1$ for $x\in[-1,1]$. From this definition, we see why $\sgn(\phi_0)$ appears in the Eikonal equation: the viscosity solution of $|\nabla\phi|-f=0$ can only have kinks pointing up while the viscosity solution of $f-|\nabla\phi|=0$ can only have kinks pointing down. If we write the Eikonal equation as $|\nabla\phi|=f$, for $\Gamma$ that is not convex, $\phi$ that is negative outside $\Gamma$ may have kinks pointing down and this $\phi$ is not a viscosity solution to $|\nabla \phi|=f$. A common error in the numerical literature is that the signed distance function associated with a non-convex curve is treated as a viscosity solution of $|\nabla\phi|=1$.

It is natural to decompose the Eikonal equation, \eqref{eq:eikonal}, into interior ($x\in U$) and exterior ($x \in \mathbb{R}^2\setminus \bar{U}$) problems, and to piece the two solutions together. If $f$ is continuous at $\Gamma$, which it is as we will show, then the interior and exterior solutions together form a viscosity solution over the entire domain (since the equation is also then satisfied on the interface $\Gamma$). We therefore focus on the interior and exterior problems separately.

The interior problem has been solved by other authors for continuous $f$ with $\inf f>0$, and existence and uniqueness have been established \cite{cl83, Ishii87, koike04}. An integral representation of $\phi$ is given by:
\begin{gather}
\phi(x)=\inf_{\gamma\in \mathscr{C}}\Big\{
\int_0^Lf(\gamma(s))ds \,\Big|\gamma(0)=x, \gamma(L)\in\Gamma \Big\},
\end{gather}
(see  \cite{lions82, ostrov00}), where $\mathscr{C}$ is the space of absolutely continuous self-avoiding curves, $s$ is the arc-length parameter, and $L$ is the total arc-length (which depends on $\gamma$). 

In the exterior problem, however, even if $f$ is continuous and $\inf f>0$, uniqueness is not guaranteed. For example, both $\phi_1(x)=||x|-1|$ and $\phi_2(x)=1-|x|$ where $x\in\mathbb{R}$ are viscosity solutions for $\sgn(1-|x|)(|\phi'(x)|-1)=0$. This motivates the following definition:\\
\begin{definition}\label{def:vs2}
The ``proper'' viscosity solution of the Eikonal equation, \eqref{eq:eikonal}, is defined to be the pointwise limit of $\phi_n$ as $n\to\infty$, where $\phi_n$ is the viscosity solution satisfying \eqref{eq:eikonal} in the sense of Definition \ref{def:vs1} in $\{|x|<n\}$ with $\phi_n(|x|=n)=0$ for $n\in \mathbb{Z}$ and $n>\max_{y\in\Gamma}d(0,y)$. 
\end{definition}\\

Here $d(E_1,E_2)=\inf_{x\in E_1,y\in E_2}\|x-y\|$ is the distance between two sets $E_1$ and $E_2$. Under this definition, for continuous $f>0$, $\phi$ is the limit of a sequence of interior problems and is therefore uniquely determined (see the previous discussion). Such a $\phi$ is also a viscosity solution in the general sense of Definition \ref{def:vs1}. A limit of the integral representation of $\phi_n$ as $n\to \infty$ reveals the viscosity solution for all $x\in\mathbb{R}^2$ with continuous $f>0$, so that a representation of $\phi(x)$ for all $x$ is given by
\begin{gather}\label{eq:valuefunc}
\phi(x)=\sgn(\phi_0)\inf_{\gamma\in \mathscr{C}}\Big\{
\int_0^Lf(\gamma(s))ds\, \Big|\gamma(0)=x, \gamma(L)\in\Gamma\Big \}.
\end{gather}

Unfortunately, while continuous on $\Gamma$, it is not guaranteed that $f$ obtained by the transport equation, \eqref{eq:extension}, is continuous in the whole domain. Before proceeding any further, we must therefore understand the nature of $f$ obtained using \eqref{eq:extension}.

\subsubsection{The cost function} Due to the method for extending $f$ into the whole domain using \eqref{eq:extension}, the behavior of the cost function $f$ is intimately linked to the behavior of $\varphi$. Aujol \& Aubert \cite{ai02} have shown that the viscosity solution of \eqref{eq:dis} that satisfies $\varphi(x\in\Gamma, \tau)=0$ converges to the steady solution, which is the proper viscosity solution (see Definition \ref{def:vs2}) of $\sgn(\phi_0)(|\nabla\varphi|-1)=0$. By \eqref{eq:valuefunc}, since the cost function in this case is $1$, $\varphi$ is the signed distance function:
\begin{gather}
\varphi(x)=\begin{cases} d(x,\Gamma) &  x\in \bar{U},\\ -d(x,\Gamma) & x\in\mathbb{R}^2\setminus\bar{U}.\end{cases}
\end{gather}
 $\varphi$ is therefore 1-Lipschitz continuous, and hence differentiable almost everywhere (a.e.). Of particular importance is the singular set of $\varphi$, which is most conveniently uncovered by studying a projection to the interface:\\

\begin{definition} \label{def:locus}
$Px=\{y\in \Gamma | d(x,\Gamma)=d(x,y)\}$ is the nonempty projection of $x$ onto $\Gamma$. Let $A=\{x | \#Px\ge2\}$ be the set of points for which the distance is achieved at multiple boundary points. The part of $A$ inside of $\Gamma$ is called the medial axis, and its closure $\bar{A}$ is called the cut locus. The skeleton $S$ is the set of centers of maximal circles (with order defined by inclusion) inside $\Gamma$.
\end{definition}\\

First note that by the $C^1$ assumption on $\Gamma$ we have that the distance between the cut locus and the interface is always positive, $d(\bar{A}, \Gamma)>0$. For $x\notin \bar{A}\cup\Gamma$, we have that
\begin{gather}
\nabla d(x,Px)=\frac{x-Px}{|x-Px|},
\end{gather}
since $Px$ and $d(x, Px)$ are both differentiable at such a point, and since $P(tx+(1-t)Px)=Px$ for $0<t<1$. Therefore, $\nabla\varphi=\sgn(\phi_0)\nabla d(x, Px)$ is continuously extended to $\Gamma$, and thus $\varphi$ is $C^1$ at $\Gamma$. Moreover, the line $x-Px$ is a characteristic line of $\varphi$ due to its alignment with $\nabla d$. Therefore, $f$ is constant along the line $x-Px$ by \eqref{eq:extension}; in addition, $f(x,\tau)$ is steady when $\tau>d(x, \Gamma)$ and $f(x)=\chi(Px)$. $f$ is continuous at $x$ since $Px$ is. $f$ is thus continuous for all $x\notin\bar{A}$ and is given by $f(x)=\chi(Px)$.

Having shown $f$ to be continuous everywhere outside of the cut locus, we are left now to explore $x\in \bar{A}$. It is well known that $\varphi(x)=\sgn(\phi_0)d(x,Px)$ is not differentiable at any point in $A$. Due to the importance of the result for studying our method, we provided an alternative proof in the appendix for the reference. Also note that $A\cap U\subset S\subset\bar{A}\cap U$ \cite{lieutier04}.  By the assumptions on $\Gamma$, the curvature $\kappa$ of $\Gamma$ exists except at possibly a finite number of points, and even at these points the left and right limits of the curvature exist; thus $\sup|\kappa|<\infty$.  When $U$ is convex, this provides an estimate for the distance between the cut locus and the interface: $d(\bar{A}, \Gamma)\ge\inf\{1/\kappa\}$ (the proof can be found in the Appendix A; the chosen convention is that the curvature of a circle is positive). In general, however, there is no estimate for $d(\bar{A}, \Gamma)$. To proceed, we must further investigate the structure of the cut locus $\bar{A}$. To this end, we observe the following:\\
\begin{lemma}\label{lemma:proj}
Let $x\in\mathbb{R}^2$. For any $y\in Px$, let  $z(t)=ty+(1-t)x$, $0<t\le 1$. Then, $Pz(t)=\{y\}$ and $z(t)\notin \bar{A}$. 
\end{lemma}
\begin{proof}
The only nontrivial part of this claim is that $z(t)\notin \bar{A}\setminus A$. Suppose $z(t)\in\bar{A}\setminus A$ for some $t$. Then $z(t)$ is the center of the osculating circle of $\Gamma$ at $y$, and the circle centered at $x$ with radius $d(x,y)$ contains strictly a part of $\Gamma$ inside, contradicting the fact that $y\in Px$. 
\end{proof}

\begin{rmk}
Based on this lemma and the assumptions on the boundary curve, we are able to get another interesting geometric property of the set $\bar{A}$: there are no cusps in $\bar{A}$ with zero angles. The proof is provided in the Appendix A. 
\end{rmk}\\

From Lemma \ref{lemma:proj}, we have
that the points in $\bar{A}$ are the terminal points of propagation along the characteristic lines of $\varphi$. Since the characteristics of $\varphi$ meet at the cut locus, $f$ may not be well defined there, so we define
\begin{gather}\label{eq: f_limit}
f(x)=\inf_{y\in Px}\chi(y), \ \ x\in \bar{A}.
\end{gather}

Here notice $f$ is extended to both the interior and exterior of the interface, which means we need to discuss the points both inside and outside of $\Gamma$.

Let's first concentrate on the part inside of $\Gamma$. One may notice that from Theorem 6.2 and Corollary 7.1 of \cite{ccm97}, we have $\bar{S}=\bar{A}\cap U$ is simply connected and the union of finitely many points and finitely many open locally analytical curves. Moreover, for every point on these open locally analytical curves, it has projection of size exactly $2$.
 
Then, let us consider the point outside of $\Gamma$. Suppose $U$ is convex, then $\bar{A}\cap U^c=\emptyset$. Otherwise, $\bar{A}\cap U^c$ may contain curves extending to infinity in general. Let's call  $U_n=B_n\cap \bar{U}^c$ with $B_n=\{|x|<n\}$. Then the Domain Decomposition Lemma in \cite{ccm97} tells us the closure of the skeleton $S_n$ of $U_n$ agrees with the skeleton of $\bar{A}\cap U^c$ inside, for instance, $B_{n/3}$. This means for the points outside of $\Gamma$ but inside of any bounded domain, by the result in \cite{ccm97} we can also get the similar characterization as the points inside of $\Gamma$. In summarize, we have
\begin{lemma}\label{lmm:A}
The set $\bar{A}\cap U$ is simply connected, consisting of finitely many points and open curves that are locally analytical, in addition, every point on those curves has a projection of size $2$. Meanwhile, in any bounded domain the set $\bar{A}\cap U^c$ consists of finitely many points and open curves that are locally analytical as well.
 \end{lemma}\\



For a point in $\bar{A}$ with a projection of size 2, function $f$, defined by (\ref{eq: f_limit}), has limiting values from both sides of the curve (for a proof see the Appendix B). We now have the following theorem, which will be important in investigating existence and uniqueness of the viscosity solutions to the Eikonal equation and the reinitialization equation:\\

\begin{theorem}\label{thm:f}
The function $f$ in (\ref{eq: f_limit}) is bounded by $c_1$ and $c_2$ and is continuous outside the cut locus $(\mbox{in } \mathbb{R}^2\setminus\bar{A})$, and the cut locus is well-separated from the interface, $d(\bar{A},\Gamma)>0$. Except at finitely many points, $f$ has a limit when $x$ approaches $\bar{A}$ from one side of $\bar{A}$.
\end{theorem}\\

\subsubsection{Viscosity solutions with a discontinuous cost function}
We may now investigate the solution of the Eikonal equation when the cost function $f$ is discontinuous with properties described in Theorem \ref{thm:f}. The solution is the steady solution of the reinitialization equation and is hence the desired level set function. Existence has been proven in fact for a much broader class of cost functions \cite{de04}. Uniqueness, however, is more challenging. Uniqueness of the Eikonal equation has been shown for cost functions $f$ satisfying certain conditions not applicable to the present case \cite{ostrov00,soravia02, soravia06, de04, soravia2000optimal}, so we must develop uniqueness of the solution particular to the cost function $f$ of the form in Theorem \ref{thm:f}. 

To begin, since $f$ is continuous on $\Gamma$, we can split \eqref{eq:eikonal} into interior and exterior problems, as discussed in \S\ref{sec:Viscosity solutions of the Eikonal equation}. We first consider the interior problem,
\begin{gather}
\begin{array}{c}
|\nabla \phi|-f(x)=0, \ \ \text{if}\ x\in U,\\
\phi(x\in\Gamma)=0.\end{array}
\end{gather}  

We have here that $H_*(x,p)=(|p|-f)_*=|p|-f^*$ and $H^*(x,p)=(|p|-f)^*=|p|-f_*$, where the sup- and inf- envelopes were defined in \S\ref{sec:Viscosity solutions of the Eikonal equation}. By Theorem \ref{thm:f}, $0< c_1\leq f^*, f_*\leq c_2$. $f^*$ is upper semi-continuous (USC) and $f_*$ is lower semi-continuous (LSC). If $f$ is continuous at $x$, $f^*(x)=f_*(x)$ and both $f^*,f_*$ are continuous at $x$. By the way we define $f$ in \eqref{eq: f_limit}, we have $f=f_*$. It is simple to show if $\phi$ is differentiable at $x_0$, that $f_*(x_0)\leq |\nabla\phi|(x_0)\leq f^*(x_0)$ (see \cite{evans10}). This implies that if $\phi$ is Lipschitz (thus differentiable a.e.), then $|\nabla \phi|=f$ outside a set of measure zero.

By approximating $f$ by its sup- and inf-convolutions (see \cite{de04}):
\begin{gather}
f^{\e}(x)=\esssup_{y}\{f(y)-|y-x|^2/\e\},\ \ \ \ 
f_{\e}(x)=\essinf_y\{f(y)+|y-x|^2/\e\},
\end{gather} 
two viscosity solutions of the Eikonal equation \eqref{eq:eikonal} are found:
\begin{gather}
\phi_M(x)=\inf_{\gamma\in \mathscr{C}}\Big\{
\int_0^Lf^*(\gamma(s))ds \Big|\gamma(0)=x, \gamma(L)\in\Gamma \Big\},\label{eq:maxsol}\\
\phi_m(x)=\inf_{\gamma\in \mathscr{C}}\Big\{
\int_0^Lf_*(\gamma(s))ds \Big|\gamma(0)=x, \gamma(L)\in\Gamma \Big\}.\label{eq:minsol}
\end{gather}
The proof for the existence is quite routine \cite{de04}. For the convenience of the readers, we provided a detailed proof in the appendix. Note that $f^*$ and $f_*$ are integrable on every curve $\gamma\in\mathscr{C}$. Also in the appendix, it is shown that $\phi_M$ and $\phi_m$ are Lipschitz continuous with the Lipschitz constant $c_2$. These two solutions are the maximal and minimal viscosity solutions of the Eikonal equation \cite{soravia02, soravia06}. The viscosity solution is unique if $\phi_M=\phi_m$. It is clear that $\phi_M=\phi_m$ if and only if for every point $x\in U$, there's a sequence of curves $\gamma_n\in \mathscr{C}$ with total length $L_n$ (where $\gamma_n(0)=x$ and $\gamma_n(L_n)\in\Gamma$) such that
\begin{gather}\label{cond:unique}
\phi_m(x)=\lim_{n\to\infty}\int_0^{L_n}f_*(\gamma_n(s))ds \ \ \mbox{and} \ \ \lim_{n\to\infty}\int_0^{L_n}(f^*-f_*)(\gamma_n(s))ds=0 .
\end{gather}
The condition implies that $\phi_M(x)\le\phi_m(x)$; hence the two are equal. Conversely, if $\phi_M(x)=\phi_m(x)$, the condition is a straightforward  corollary of the definition.

The proof of the uniqueness in the literature can't be applied to the discontinuous cost function $f$ of Theorem \ref{thm:f}. However, we are able to  check \eqref{cond:unique} directly. The proof is provided in the appendix {\bf B}. This proof for uniqueness is new and it might be modified to prove the uniqueness of the Eikonal equation with a class of discontinuous cost functions. In the exterior problem, uniqueness of the proper viscosity solution may be proved by first considering the finite domain $\bar{U}^c\cap B_n$ with $B_n=\{|x|<n\}$, where $\phi_n$ satisfying $\phi_n(x)=0$, at $|x|=n$ using a similar proof as in the interior problem and then taking $n\to\infty$. Recalling that $f=f_*$, we finally have the desired result:\\

\begin{theorem}\label{thm:37}
The proper viscosity solution to \eqref{eq:eikonal} with the cost function obtained from \eqref{eq:extension} is unique and is given by \eqref{eq:valuefunc}. It is hence $c_2$-Lipschitz continuous and $C^1$ in $\mathbb{R}^2\setminus \bar{A}$.
\end{theorem}\\

\subsection{The reinitialization equation}\label{sec: reinit}
Finally, we show that the reinitialization equation has viscosity solutions, and the solution that is zero on $\Gamma$ is unique and converges to the proper viscosity solution of the Eikonal equation \eqref{eq:eikonal}. Recall the reinitialization equation, written more generally as
\begin{gather}\label{eq:reinitialization}
\begin{array}{c}
\displaystyle\frac{\partial u}{\partial\tau}+\sgn(u_0)(|\nabla u|-g(x))=0,\vspace{.1in}\\
u(x,0)=u_0(x),\end{array}
\end{gather}
where $u_0\in UC(\mathbb{R}^2)$, with $UC$ the class of uniformly continuous functions. Here $u$ could be $\varphi$ or $\phi$ and $g$ could be $1$ or $f$, and the Hamiltonian is written as $H(x, p)=\sgn(u_0(x))(|p|-g(x))$. We assume that $0<c_1\le g\le c_2$ and that $\Gamma$ is the zero level set of $u_0$.

For time-dependent Hamilton-Jacobi equations, the classical solutions are not well-defined beyond the intersection of characteristics. For some applications, the multi-valued solutions are important \cite{jlot05}; for our purpose, we need the viscosity solution, whose definition is similar to the one for the Eikonal equation in \eqref{def:vs1}, with the only difference being the addition of a time derivative.

Generally, if $g$ is not continuous, we can again approximate $g$ by $g^\e$ and $g_\e$ and take the limit $\varepsilon\to0$ as we did for the Eikonal equation. Therefore, we first consider the case where $g$ is continuous. Motivated by the solutions of the Eikonal equation and the solution provided in \cite{ai02} for $g=1$, we construct the formula of the solution,
\begin{gather}\label{eq:solhj}
u(x,\tau)=\begin{cases}
\sgn(u_0)\inf_{\gamma\in\mathscr{C}}\{|u_0(\gamma(\tau))|+\int_0^\tau g(s)ds\, |\,\gamma(0)=x \}
&\tau\le\tau_x,\vspace{.1in}\\
\sgn(u_0) \inf_{\gamma\in\mathscr{C}}\{\int_0^Lg(s)ds \,|\, \gamma(0)=x, \gamma(L)\in \Gamma\}
& \tau>\tau_x,\end{cases}
\end{gather}
where $\tau_x$ is given by
\begin{multline}
\tau_x=\inf\Bigg\{\bar{\tau}\geq d(x,\Gamma) \Big| \forall \tau>\bar{\tau}:  \inf_{\gamma\in\mathscr{C}}\Big\{|u_0(\gamma(\tau))|+\int_0^\tau g(s)ds\ \ |\ \ \gamma(0)=x \Big\}\\
>\inf_{\gamma\in\mathscr{C}}\Big\{\int_0^Lg(s)ds\ \ |\ \ \gamma(0)=x, \gamma(L)\in \Gamma \Big\} \Bigg\}.
\end{multline}
It is evident that $\tau_x$ is continuous on $x$ and $\tau_x\leq c_2\,d(x, \Gamma)/c_1$. This formula is closely related to the Lax-Hopf formula and $u$ here has the interpretation of the value function with cost $g$ (see \cite{evans10}). At $\tau_x$, we must have the two expressions in \eqref{eq:solhj} being equal, for otherwise, the first is always strictly larger than the second for all $\tau\geq d(x,\Gamma)$ but this can't be true at $\tau=L_{opt}\geq d(x,\Gamma)$, where $L_{opt}=\liminf_{n\to\infty}L_n$ and $L_n$ is a sequence such that $\exists \gamma_n,\gamma_n(0)=x,\gamma_n(L_n)\in\Gamma$, $\lim_{n\to\infty}\int_0^{L_n}g(\gamma_n(s))ds
=\inf_{\gamma\in\mathscr{C}}\int_0^Lg(\gamma(s))ds$. In addition, we find that $d(x,\Gamma)\leq L_{opt}\leq \tau_x\leq c_2d(x, \Gamma)/c_1$. \\

\begin{rmk}
If one were to define a simpler time,
\begin{multline}
\tau_x=\inf\Bigg\{\tau \Big|
\inf_{\gamma\in\mathscr{C}}\Big\{ |u_0(\gamma(\tau))|+\int_0^\tau g(s)ds\ \ |\ \ \gamma(0)=x \Big\}>\\
\inf_{\gamma\in\mathscr{C}}\Big\{\int_0^Lg(s)ds\ \ |\ \ \gamma(0)=x, \gamma(L)\in \Gamma\Big\}\Bigg\},
\end{multline}
then $\tau_x$ might be smaller than $L_{opt}$. When $u(x,\tau)$ is given by the second formula, there could be no paths with lengths less than $\tau+\e_n,\e_n\to0$ to approximate the infimum. For $\tau$ bigger than so-defined $\tau_x$, the first expression might be smaller than the second one. The dynamic programming principle in the appendix then cannot be shown. 
\end{rmk}\\

The two expressions given in \eqref{eq:solhj} are continuous in both $x$ and $\tau$ and they give the same value at $\tau=\tau_x$. $u$ is then continuous in both $x$ and $\tau$. We can also see that it satisfies the initial and boundary conditions of the reinitialization equation. In the appendix, we verify that $u$ is a viscosity solution. From the formula, since $\tau_x$ is bounded by $c_2d(x,\Gamma)/c_1$, we see that the solution on any compact set converges to the proper viscosity solution of the Eikonal equation \eqref{eq:valuefunc} in finite time.

Uniqueness of the solution may be shown under the assumption that $u(x\in\Gamma,\tau)=0$, which can be ensured numerically. Following \cite{ai02}, consider
\begin{gather}\label{prob:in}
\displaystyle\frac{\partial u}{\partial\tau}+(|\nabla u|-g)=0 \ \ \ x\in U,\vspace{.05in}\\ 
u(x,0)=u_0(x) \ \ \  x\in\bar{U}\nonumber \vspace{.05in}\\
u(x\in\Gamma, \tau)=0,\nonumber
\end{gather}
and
\begin{gather}\label{prob:out}
\displaystyle\frac{\partial u}{\partial\tau}-(|\nabla u|-g)=0 \ \ \  x\in \mathbb{R}^2\setminus\bar{U},\vspace{.05in}\\ 
u(x,0)=u_0(x) \ \ \  x\in\mathbb{R}^2\setminus U\nonumber\vspace{.05in}\\
u(x\in\Gamma, \tau)=0.\nonumber
\end{gather}
The uniqueness of the solutions for these two problems have been established if $g$ is continuous and bounded below by a positive number \cite{Ishii86}.  This is enough to say that there is at most one viscosity solution satisfying $u(x\in\Gamma, \tau)=0$. One common mistake in the literature is to assume that, since $\sgn(u_0(x\in\Gamma))=0$, that $u_{\tau}(x\in\Gamma)=0$ by \eqref{eq:reinitialization}. However, this argument is inadequate in the viscosity sense, since the set of viscosity solutions are unchanged by a redefinition of the sign function to a value $\sgn(0)\in[-1, 1]$. It is sensible, however, that all viscosity solutions should have $u(x\in\Gamma,\tau)=0$ as the characteristics flow out of $\Gamma$, and fortunately we can develop numerical schemes to ensure $u(x\in\Gamma,\tau)=0$. 

Finally, for $g$ equal to the cost function $f$ obtained from \eqref{eq:extension}, it may be discontinuous as previously discussed. As was done for the Eikonal equation, approximating $f$ with $f^{\varepsilon}$ and $f_{\varepsilon}$, and taking $\varepsilon\to 0$, we have the maximal and minimal solutions with $g$ replaced by $f^*$ and $f_*$. And as for the Eikonal equation, these two are equal and the solution that is zero on the interface $\Gamma$ is unique: \\

\begin{theorem}
Assume in \eqref{eq:reinitialization} that either $g=1$ or $g=f$ (the cost function obtained using \eqref{eq:extension}). The viscosity solution that satisfies $u(x\in\Gamma)=0$ is unique and is provided in \eqref{eq:solhj}. This solution converges to the proper viscosity solution of \eqref{eq:eikonal}.
\end{theorem}\\

\section{Numerical schemes}\label{sec:numerics}
We have shown in theory that the LGPR method yields the desired level set function. We now proceed to describe numerical schemes for solving the PDEs introduced in \S\ref{sec:setup}, with a few modifications from classical methods \cite{pmozk99,min10}. We also show how subcell resolution may be used to extend the interface gradient away from the surface with high accuracy. First we present a numerical scheme for solving the transport equation which involves a second-order accurate upwind Essentially Non-Oscillatory (ENO) scheme with subcell resolution in space and Gauss-Seidel iteration in time, and then we describe a method for solving the reinitialization equation which involves a Godunov numerical Hamiltonian scheme in space and again Gauss-Seidel iteration in time.

\subsection{Numerical setup}

Consider an Eulerian grid with uniform grid size $h$ upon which the interface $\Gamma$ is overlaid. Gridpoints $(x_i, y_j)$ are defined by $x_i=ih$ and $x_j=jh$, with $i$ and $j$ taking integer values. (Unlike in the theoretical part of the paper above, in which $x$ and $y$ correspond to two points in $\mathbb{R}^2$, in the remainder of the paper $x$ and $y$ are coordinates, $(x,y)\in \mathbb{R}^2$). 


In order to approximate derivatives of possibly non-smooth functions we will rely on ENO finite differences (see \cite{os91}). In addition, in the solution of Hamilton-Jacobi equations, one-sided (upwind) derivatives are commonly used to retain causality (i.e. information follows the characteristics). In this paper, the following one-sided second order ENO finite differences will be used to approximate first derivatives,
\begin{gather}
\begin{array}{c}
D_x^-\phi_{i,j}=\displaystyle\frac{\phi_{i,j}-\phi_{i-1,j}}{h}+\frac{h}{2}\minmod(D_{xx}\phi_{i,j}, D_{xx}\phi_{i-1,j}),\vspace{.1in}\\
D_x^+\phi_{i,j}=\displaystyle\frac{\phi_{i+1,j}-\phi_{i,j}}{h}-\frac{h}{2}\minmod(D_{xx}\phi_{i,j}, D_{xx}\phi_{i+1,j}),\vspace{.1in}\\
\minmod\{a,b\}=\left\{
\begin{array}{cc}
0& \text{if}\ ab<0,\\
\sgn(a)\min\{|a|, |b|\}& else,
\end{array}\right.
\end{array}
\end{gather}
where the second derivative is given by the centered difference formula
\begin{gather}
D_{xx}\phi_{i,j}=\frac{1}{h^2}\left(\phi_{i+1,j}-2\phi_{i,j}+\phi_{i-1,j}\right).
\end{gather} 

\subsection{The transport equation}


Recall the definitions of the stretch function $\chi(x)=|\nabla\phi|(x)$ on $x\in \Gamma$ and the inward pointing unit normal vector $\hat{n}=\nabla\phi/|\nabla\phi|$. In solving the transport equation, we aim to accurately preserve the stretch function as well as its tangential derivative along the interface,
\begin{gather}
\hat{n}\cdot\nabla\nabla\phi\cdot(I-\hat{n}\hat{n})|_{\Gamma}=D_s\chi(\Gamma(s))\s,
\end{gather}
where $s$ is the arc-length and $\s(s)=\Gamma'(s)$ is the unit vector tangent to the surface $\Gamma$. Our strategy will be to preserve the zero level set of $\phi$ (the location of the surface $\Gamma$) and the stretch function $\chi$ with at least second order accuracy, and thus $D_s\chi(\Gamma(s))\s$ is formally preserved with first order accuracy. 

In the solution of the transport equation we will use a subcell resolution (SR) technique to obtain the cost function $f$ (see Eq.~\eqref{eq:eikonal}). In SR, the interface is determined by interpolating the obtained signed distance function $\varphi$, computing the the gradient there, and modifying the one-sided ENO derivatives according to the interface \cite{harten89, min10}. To illustrate the subcell resolution technique, consider as an example the case $\varphi_{ij}\varphi_{i-1,j}\leq0$. Letting 
\begin{gather}
a_{ij}=h^2 \minmod(D_{xx}\varphi_{ij}, D_{xx}\varphi_{i-1,j}),
\end{gather}
and assuming $(x_{\Gamma}, y_j)\in\Gamma$, $x_{\Gamma}$ is found by quadratic interpolation at $\varphi_{i-1,j},\varphi_{ij}$ using the second order derivative $a_{ij}/h^2$. The approximations of the first derivatives are then given by
\begin{gather}\label{formula:gradinterface}
\begin{array}{c}
\partial_x\phi_0(x_{\Gamma}, y_j)\approx \displaystyle\frac{\delta^{x-}_{ij}}{h}D_x^0\phi_{0,i-1,j}+\frac{h-\delta^{x-}_{ij}}{h}D_x^0\phi_{0,ij},\vspace{.1in}\\
\partial_y\phi_0(x_{\Gamma}, y_j)\approx \displaystyle\frac{\delta^{x-}_{ij}}{h}D_y^0\phi_{0,i-1,j}
+\frac{h-\delta^{x-}_{ij}}{h}D_y^0\phi_{0,ij},
\end{array}
\end{gather}
where $D_x^0, D_y^0$ are the centered differences and 
\begin{gather}\label{formula:interface}
\delta^{x-}_{ij}=x_i-x_{\Gamma}=h\left(\frac{1}{2}+\frac{(\varphi_{ij}+\varphi_{i-1,j})-a_{ij}/4}
{(\varphi_{ij}-\varphi_{i-1,j})+\sgn(\varphi_{ij}-\varphi_{i-1,j})\sqrt{D_{ij}}}\right), 
\end{gather}
where
\begin{gather}
D_{ij}=(a_{ij}/2-\varphi_{ij}-\varphi_{i-1,j})^2-4\varphi_{ij}\varphi_{i-1,j}.
\end{gather}


Having now obtained the cost function on the interface, $f(x_{\Gamma}, y_j)=\sqrt{\partial_x\phi_0^2+\partial_y\phi_0^2}$, the left ENO derivative of $f$ at $(i, j)$ is modified as 
\begin{gather*}
D_x^-f_{ij}=
\frac{f_{ij}-f(x_{\Gamma}, y_j)}{\delta^{x-}_{ij}}+\frac{\delta^{x-}_{ij}}{2}\minmod(D_{xx}f_{i-1,j}, D_{xx}f_{ij}),
\end{gather*} 
while $D_x^+$ at $(i-1, j)$ is similarly modified. 

Next, the cost function $f$ is extended into space by solving the transport equation,
\begin{gather}
\frac{\partial f}{\partial \tau}+\nabla d\cdot\nabla f=0,
\end{gather} 
where $\nabla d=\sgn(\varphi)\nabla\varphi$. We will denote $\sgn(\varphi_{ij})D_x^{\pm}\varphi_{ij}$ by $D_x^{\pm}d_{ij}$ (where $D_x^{\pm}$ are acting on $\varphi$ not $|\varphi|$ and $D_x^{\pm}$ are modified with subcell resolution near the interface) and we define
\begin{gather}
D_xd_{ij}=\text{maxabs}\{\max\{D_x^-d_{ij}, 0\},\ \min\{0, D_x^+d_{ij}\}\},
\end{gather}
where $\text{maxabs}\{a,b\}=(a-b)\mathbbm{1}_{\{|a|\ge |b|\}}+b$. For the numerical gradient we then take
\begin{gather}
\eta_{ij}^x=\displaystyle\frac{D_xd_{ij}}{\sqrt{(D_xd_{ij})^2+(D_yd_{ij})^2+\e^2}},\ \ \  \eta_{ij}^y=\displaystyle\frac{D_yd_{ij}}{\sqrt{(D_xd_{ij})^2+(D_yd_{ij})^2+\e^2}},
\end{gather}
where $\e$ is a small parameter (chosen here to be $10^{-7}$) to avoid the case that both $D_x$ and $D_y$ are close to zero at potentially irregular points. $D_xd_{ij}$ is so chosen to ensure that the information propagating to $(i,j)$ is coming points closer to the interface, which follows the correct characteristic directions, and also that any oscillation in $f$ on the the cut locus $\bar{A}$ (see Def.~\ref{def:locus}), is suppressed. Using the definitions above, a complete spatial discretization for the transport equation is given by
\begin{gather}
L_T(f_{ij})=-\left((\eta_{ij}^x)^+D_x^-f_{ij}-(\eta_{ij}^x)^-D_x^+f_{ij}+(\eta_{ij}^y)^+D_y^-f_{ij}-(\eta_{ij}^y)^-D_y^+f_{ij}\right),
\end{gather}
where $a^+=\max\{a, 0\}$ and $a^-=-\min\{a, 0\}$. 

We now turn to the discretization of time, $\tau$, which is not real time but merely a parameter used to relax the system to its steady state. Different timesteps $k_{ij}$ are chosen for different points to ensure stability. Since $\delta_{ij}^{x\pm}$ or $\delta_{ij}^{y\pm}$ can be small, a Courant-Friedrichs-Lewy (CFL) condition for convergence requires that the time step be small near the interface. We use the same convention as in \cite{min10},
\begin{gather}\label{eq:kij}
k_{ij}=C\min\{h, \delta_{ij}^{x\pm},  \delta_{ij}^{y\pm}\},
\end{gather}
where $C$ is a constant taken small enough to ensure convergence. The largest possible value of $C$ depends on the cost function, and $C<1$ is sufficient when $f=1$. We choose $C=1/2$ for the numerical examples to come, which is sufficient for the cases studied.  Further, a Gauss-Seidel iteration is employed (values are updated using the newest data along the chosen sweeping directions) to allow information to propagate long distances in some directions with one iteration. The Gauss-Seidel iteration is given by
\begin{gather}
f_{ij}\leftarrow f_{ij}+k_{ij}L_T(f_{ij})
\end{gather}
along the four sweeping directions $(i,j)=(1:N,1:N)$, $(i,j)=(1:N, N:-1:1)$, $(i,j)=(N:-1:1,1:N)$, and $(i,j)=(N:-1:1, N:-1:1)$, repeatedly.


\begin{rmk}
For the application of a local level set method, it may be preferable to use a direct time-stepping method. One possibility is the second-order TVD Runge-Kutta method,
\begin{gather}\label{tvdrk}
\begin{array}{c}
\tilde{f}_{ij}^{n+1}=f_{ij}^n+k_{ij}L_T(f_{ij}^n),\vspace{.1in}\\
f_{ij}^{n+1}=\displaystyle\frac{1}{2}\left(f_{ij}^n+\tilde{f}_{ij}^{n+1}\right)+\frac{k_{ij}}{2}L_T(\tilde{f}_{ij}^{n+1})
\end{array}
\end{gather}
\end{rmk}

\begin{rmk}
For the transport equation \eqref{eq:extension}, in many references $\nabla \varphi$ is discretized by centered difference. Near the interface, the signed distance function is $C^1$ so that this simple treatment is convenient and sufficient. However, once $f$ is extended into the domain where the signed distance function is not smooth, the scheme developed in this section is expected to be more accurate.
\end{rmk}

\subsection{The reinitialization equation}
We now introduce the Godunov Hamiltonian scheme used to solve the reinitialization equation. By a similar consideration of causality, we use the Godunov Hamiltonian $\hat{H}_{ij}(D_x^-u_{ij},D_x^+u_{ij},D_y^-u_{ij},D_y^+u_{ij})$, where $s_{ij}=\sgn(u_0(x_i,y_j))$, and
\begin{gather}
\hat{H}_{ij}(a, b,c , d)
=\left\{\begin{array}{cc}
s_{ij}\left(\sqrt{\max\{a^+, b^-\}^2+\max\{c^+, d^-\}^2}-g_{ij}\right)& \text{if}\ s_{ij}\geq 0,\vspace{.1in}\\
s_{ij}\left(\sqrt{\max\{a^-, b^+\}^2+\max\{c^-, d^+\}^2}-g_{ij}\right)& \text{if}\ s_{ij}< 0,
\end{array}\right.
\end{gather}
for spatial discretization \cite{pmozk99, min10}. (Recall $u$ may be $\varphi$ or $\phi$ while $u_0=\phi_0$.)

It would seem to be the case that the subcell resolution (SR) technique is not required in order to achieve second-order accuracy with the Godunov Hamiltonian scheme. Without SR, the absolute error $||\nabla u|-g|$ does indeed scale as $O(h^2)$ (recall that $g$ can be either $1$ or $f$, see \S\ref{sec: reinit}), but the interface location is generally not determined at the same level of accuracy, especially when the interface gradients of $u$ and $u_0$ are different. This is particularly important for \eqref{eq:dis} because the information comes from the zero set of $\varphi$ in \eqref{eq:extension}. The interface gradient would be determined only with first order accuracy and the variation in the stretch function $D_s\chi(\Gamma(s))\s$ is then poorly captured. Thus, subcell resolution is still required to achieve second order accuracy for the Godunov Hamiltonian scheme. For example, when $u_{0,ij}u_{0,i-1,j}\le0$, we modify $D_x^-$ for $(x_i,y_j)$ to
\begin{gather*}
D_x^-u_{ij}=
\frac{u_{ij}}{\delta^{x-}_{ij}}+\frac{\delta^{x-}_{ij}}{2}\minmod(D_{xx}u_{i-1,j}, D_{xx}u_{ij}).
\end{gather*} 
and $D_x^+$ for $(x_{i-1}, y_j)$ is similarly modified. For time discretization we again use Gauss-Seidel iteration using the same spatially varying timestep as in \eqref{eq:kij}. 

The numerical Hamiltonian ensures that information propagating to $(x_i, y_j)$ comes from values of $u$ closer to zero. The direction of the characteristics is preserved and $u(x\in\Gamma,\tau)=0$ is ensured. The scheme with this numerical Hamiltonian is monotone. 
While not identical to the case of present interest, that monotone schemes of the form $u_t+H(\nabla u)=0$, where $H$ is continuous, have been shown to converge to the viscosity solution \cite{cl84} is suggestive. In practice we do observe the expected convergence. 

\section{Application to Eulerian Immersed boundary method}\label{sec:app}
As we have mentioned above, our method allows the application of local level set method to Eulerian immersed boundary method.
We'll basically use the local level set method proposed in \cite{pmozk99}.  The skeleton of the algorithm is given as following:

0. Given the immersed interface $\Gamma(\xi)$, we construct the signed distance function $\varphi(x)$. Pick three positive numbers $\alpha<\beta<\gamma$.

1. Let $T=\{\b{x}: |\varphi(x)|<\gamma\}$. In the initialization step, we set $\chi(x\in\Gamma)=|\b{X}_{\xi}|$. Then, use LGPR to obtain $\phi$ in $T$. In the reinitialization step, how LGPR is used to get a new level set function has been explained in detail in the above sections.

2. Solve the fluid equations (We use Navier-Stokes equations in the computation example later) where the fluid body force is computed by \eqref{eq:F_elastic}. Evolve $\phi$ inside $T$ with $c(\phi)\b{u}$ used to convect the level set functions. We use the cutoff function introduced in \cite{pmozk99}
$$
c(y)=\left\{\begin{array}{cc}
1&\text{if}\ |y|\le \beta\\
(|y|-\gamma)^2(2|y|+\gamma-3\beta)/(\gamma-\beta)^3&\text{if}\ \beta<|y|\le\gamma\\
0&\text{if}\ |y|>\gamma
\end{array}\right..
$$

4. When the interface stays far enough from the boundary of the computation tube or $|\nabla\phi|$ at the interfaces stays in $[1-c, 1+c]$ for a given $c$, go to step 3; Else let $T_0$ be the $(\beta-\alpha)$ neighborhood of $T$. We reinitialize the level set function to the signed distance function in $T_0$ and move to Step $1$.

For this algorithm, a lot of practical issues have been omitted since they are standard or straightforward but tedious to explain. However, we would like to point out how to use LGPR to initialize the level set function $\phi$. The method is as following:

In a thin tube containing the interface, for every grid point $(x_i,y_j)$, supposing $z=P(x_i,y_j)=X(\xi_z)$ is the projection, we determine the distance by computing $d=|(x_i,y_j)-z|$ and the sign of of the signed distance function $\varphi$ by checking $((x_i,y_j)-z)\times X_{\xi}(\xi_z)$. We record $|X_{\xi}(\xi_z)|$ at the point $(x_i, y_j)$, which is the value of $|\nabla\phi_{0,ij}|$. The distance function $d(x)$ is obtained by the Fast Sweeping Method (\cite{zhao04}) and the sign is extended in a similar manner. $\varphi$ is then determined, and we evolve \eqref{eq:dis} to improve $\varphi$. The values of the cost function $f$ at the interface are interpolated from $|\nabla\phi_{0,ij}|$. Our method may then be applied to recover $\phi$.
 
\section{Numerical experiments}\label{sec:exp}
In this section we test the LPGR method with a few illustrative examples. First we show that the method achieves the expected accuracy in a setting in which the initial level set function and interface stretching are specified, and we show that the elastic force is preserved through the reinitialization process. Next, we simulate the behavior of a stretched membrane in a fluid using a local level set method, and show that reinitialization is necessary and effective when the membrane encroaches the boundary of the computational tube. Finally, we show how the method could also be used to initialize level set functions from a given parametric curve. 

\subsection{Example 1}
\begin{figure}[htbp]
\begin{center}
\includegraphics[width=0.9\textwidth]{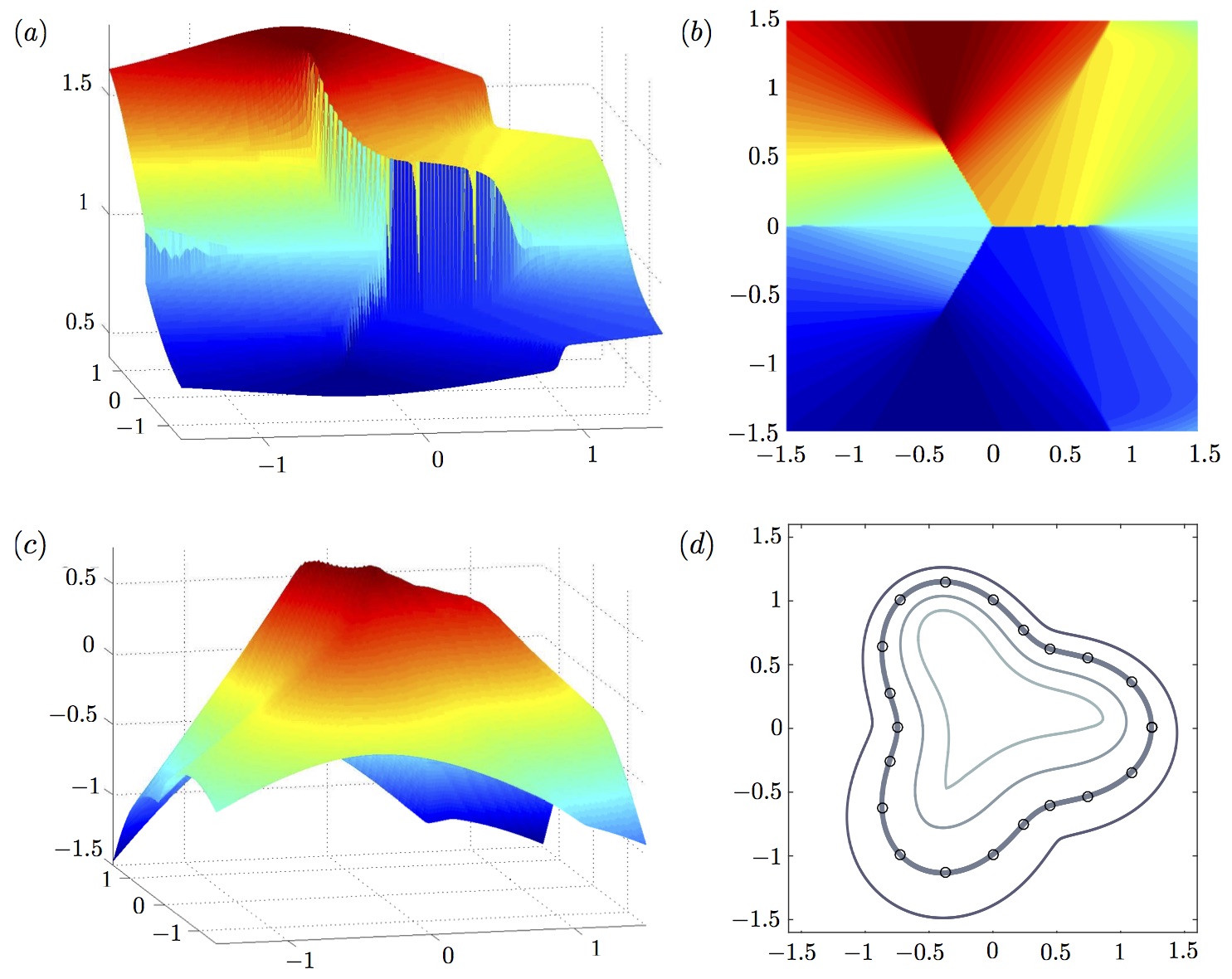}
\caption{Example 1. (a) The discontinuous cost function $f$ constructed from the specified stretch $\chi=\exp(0.5y)$ on the interface $r(\theta)=1+0.25\cos(3\theta)$. (b) The same as (a) from above. The cut locus (where $f$ is discontinuous) is more readily apparent. (c) The level set function $\phi$ after reinitialization. (d) Contour plots of the level set function after reinitialization. The bold line is the zero set corresponding to $\Gamma$ while symbols are sample positions on the interface computed analytically.}
\label{fig:discontinuous_ex1}
\end{center}
\end{figure}
We first present a typical example in which the cost function is discontinuous in order to show that the transport equation is successfully solved and that the interface gradient is preserved with the desired accuracy. Let the surface $\Gamma$ be parameterized in polar coordinates by $r(\theta)=1+0.25\cos(3\theta)$. We take as the initial level set function $\phi_0(x,y)=\varphi(x,y)\exp(0.5y)$ where $\varphi$ is the signed distance function relative to $\Gamma$. Since the interface is non-convex, characteristics of $\varphi$ from the interface intersect both inside and outside of $\Gamma$. 

Using the computational domain $[-1.5,1.5]^2$ and grid size $h=3/256$, we show in Figs.~\ref{fig:discontinuous_ex1}(a,b) two views of the cost function $f$ produced by solving the transport equation. The set where $f$ is discontinuous (the cut locus, where characteristics intersect) is well captured. The cost function does not oscillate near the cut locus using our scheme for solving the transport equation. Figs.~\ref{fig:discontinuous_ex1}(c,d) show the numerical solution of the reinitialized level set function $\phi$ and its contours, which converges rapidly even though $f$ is discontinuous. 

\begin{figure}[htbp]
\begin{center}
\includegraphics[width=\textwidth]{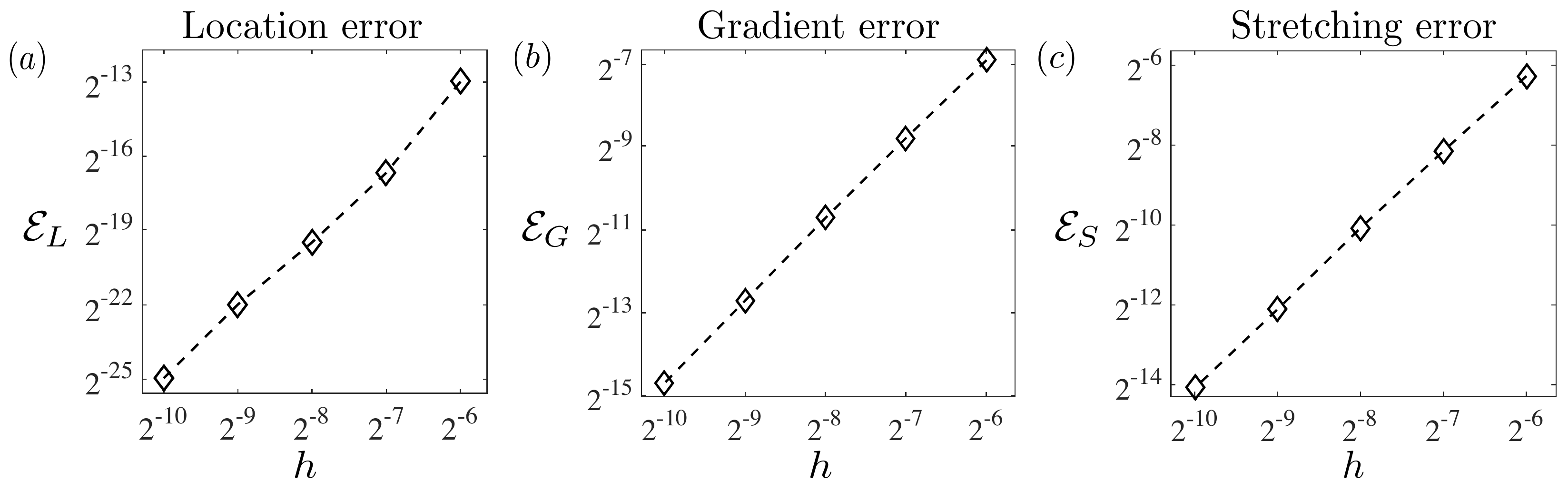}
\caption{Convergence results for Example 1. (a) The error in the interface location decreases roughly as $O(h^3)$. (b) The error in the interface gradient decays as $O(h^2)$. (c) The error in the interface stretching computed using the reinitialized level set function decays as $O(h^2)$.}
\label{fig:stretcherr}
\end{center}
\end{figure}

To test the accuracy of the numerical method, we find all the interface points $p^0=(x_{\Gamma}^0, y_j)$ or $p^0=(x_i, y^0_{\Gamma})$ by the interpolation formula \eqref{formula:interface} using the data $\phi_0$ and correspondingly points $p=(x_{\Gamma}, y_j)$ or $p=(x_i, y_{\Gamma})$ using $\phi$. The gradients $\nabla\phi_0$ and $\nabla\phi$ are computed using \eqref{formula:gradinterface} except that $\delta_{ij}^{x\pm}$ and $\delta_{ij}^{y\pm}$ are computed using the level set functions themselves instead of $\varphi$. We define the interface location error by $\mathcal{E}_L=\max\{|x_{\Gamma}^0-x_{\Gamma}|, |y_{\Gamma}^0-y_{\Gamma}|\}$, error in the interface gradient by $\mathcal{E}_{G}=\max\{|\nabla\phi_0(p^0)-\nabla\phi(p)|\}$ and the stretching error by $\mathcal{E}_{S}=\max\{||\nabla\phi(p)|-\exp(0.5y_p)|\}$. Fig.~\ref{fig:stretcherr}(a-c) show the decay rate of these errors with the spatial step size $h$. The error in the computed interface position decays roughly as $O(h^{3})$, while $\nabla\phi$ (both the direction and norm) is accurate to $O(h^{2})$ as expected. Elastic forces computed using the reinitialized level set function are therefore expected to carry over with first order accuracy, which we now probe. 

\begin{figure}[htbp]
\begin{center}
\includegraphics[width=0.9\textwidth]{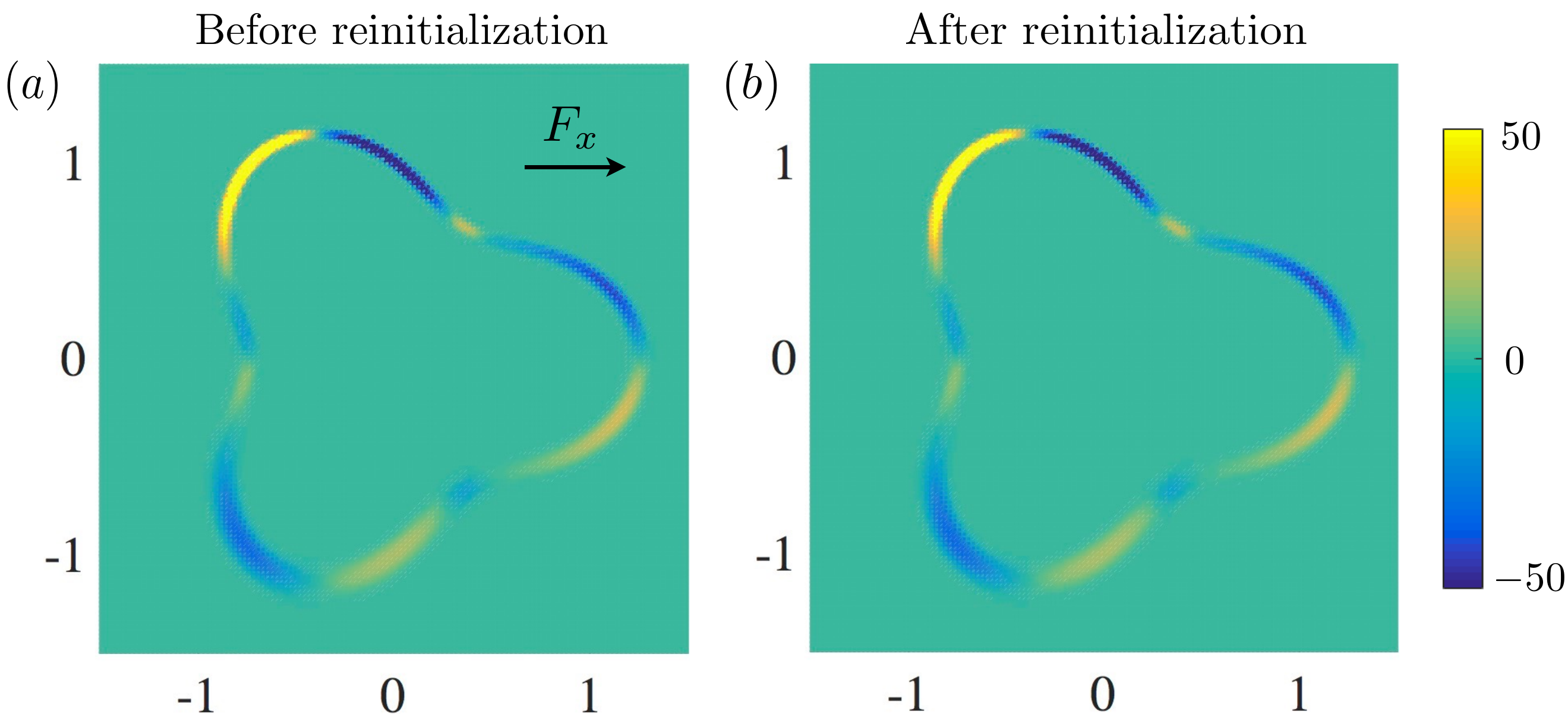}
\caption{The $x$-component of the elastic force $F_x$ in Example 1 due to the complex initial stretching, using grid size $h=3/128$, before reinitialization (a) and after reinitialization (b). The force computed has different relations with the curvature for $y>0$ and $y<0$. The traditional reinitialization method would naturally lose information about interface stretching, but here we see that the elastic force is preserved through the reinitialization process.}
\label{fig:forcechange}
\end{center}
\end{figure}

\begin{figure}[htbp]
\begin{center}
\includegraphics[width=0.9\textwidth]{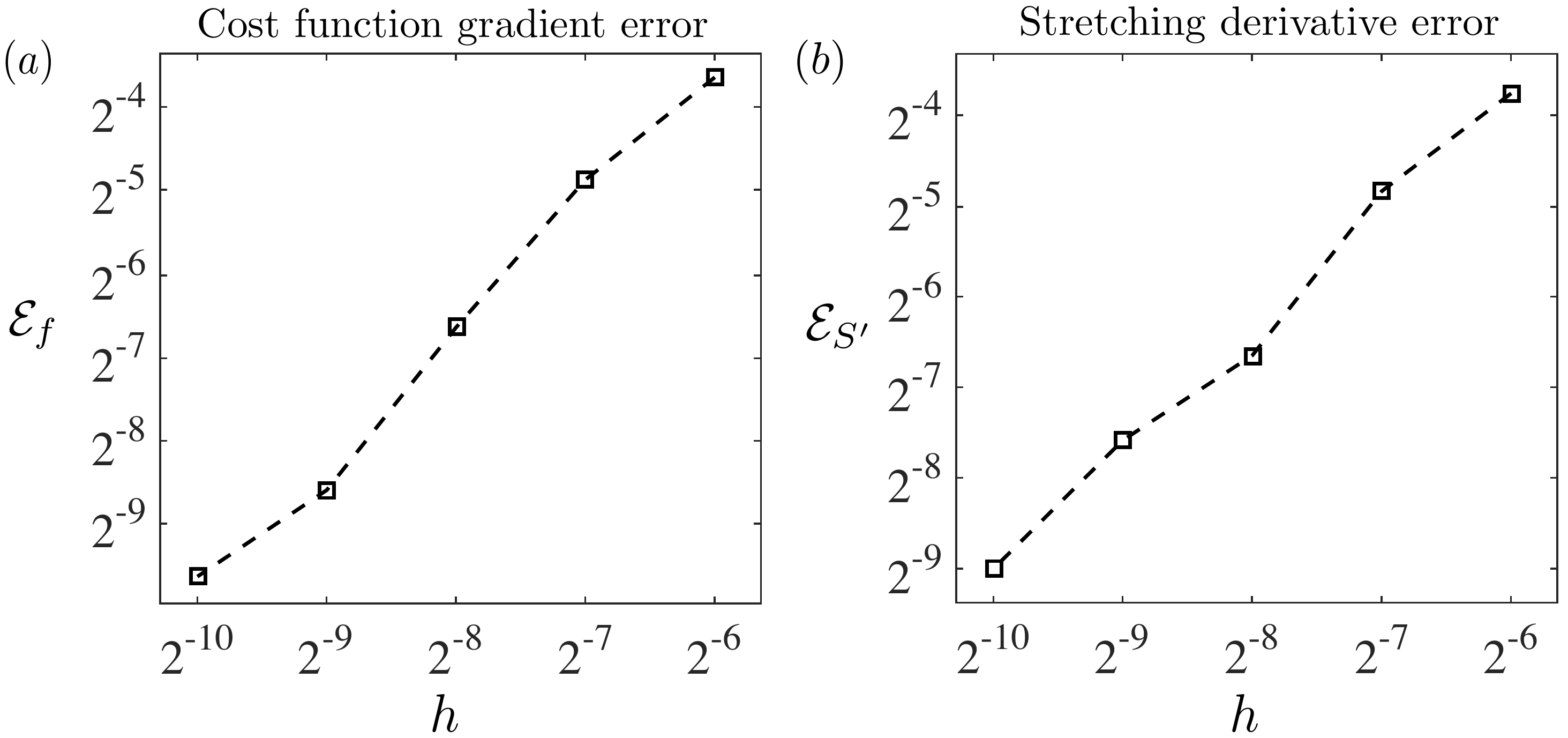}
\caption{(a) The error in the computed cost function gradient. (b) The error in the stretching derivative (the tension gradient in the case of Hookean elastic energy). Both are first order accurate as expected.}
\label{fig:elasticforce}
\end{center}
\end{figure}

Figures~\ref{fig:forcechange}(a-b) show the $x$-component of the force associated with the prescribed stretching in the first example, before reinitialization using $\phi_0$ (a) and after reinitialization using $\phi$ (b). For $y>0$, $|\nabla\phi|>1$ (the interface is in tension) while for $y<0$, $|\nabla\phi|<1$ (the interface is in compression). The force is computed on $[-1.5,1.5]^2$ with $h=3/128$, using a Hookean elastic energy $E(\chi)=(K/2)(\chi-1)^2$ (with $K=5$). Following Peskin \cite{peskin02}, in the description of the force \eqref{eq:F_elastic} that might be used in an immersed boundary context we use a smoothed $\delta$ function, $\delta_h=\mathbbm{1}_{\{|\phi|\le 3h\}}(1+\cos(\pi \phi/(3h)))/(6h)$. The traditional reinitialization method would naturally lose information about interface stretching, but here we see that the elastic force is preserved through the reinitialization process.

We now test the accuracy in computing both the tangent gradient of the cost function and the stretching derivative (the tension gradient in the case of Hookean elastic energy). We define $Tf(p^0)=\nabla f\cdot(I-\hat{n}\hat{n})(p^0)$ where $\hat{n}(p^0)=\nabla\phi^0/|\nabla\phi^0|$ and correspondingly $T\chi(p^0)=\nabla\exp(0.5 y)\cdot (I-\hat{n}\hat{n})(p^0)$. The cost function gradient error $\mathcal{E}_f=\max\{|Tf(p^0)-T\chi(p^0)|\}$ and the stretching derivative error $\mathcal{E}_{S'}=\max\{|\hat{n}\cdot\nabla\nabla\phi\cdot(I-\hat{n}\hat{n})(p)-\hat{n}\cdot\nabla\nabla\phi^0\cdot(I-\hat{n}\hat{n})(p^0)|\}$, where $\hat{n}(p)=\nabla\phi/|\nabla\phi|$ are then defined. The gradients ($\nabla f$, $\nabla\phi$ etc) are computed using \eqref{formula:gradinterface}, and $\nabla\nabla\phi(p)$ is computed by interpolating the corresponding Hessians at the two nearby points. As shown in Fig.~\ref{fig:elasticforce}, the quantities are computed with roughly first order accuracy as expected.

\subsection{Example 2}\label{sec:ex2}
\begin{figure}[htbp]
\begin{center}
\includegraphics[width=0.9\textwidth]{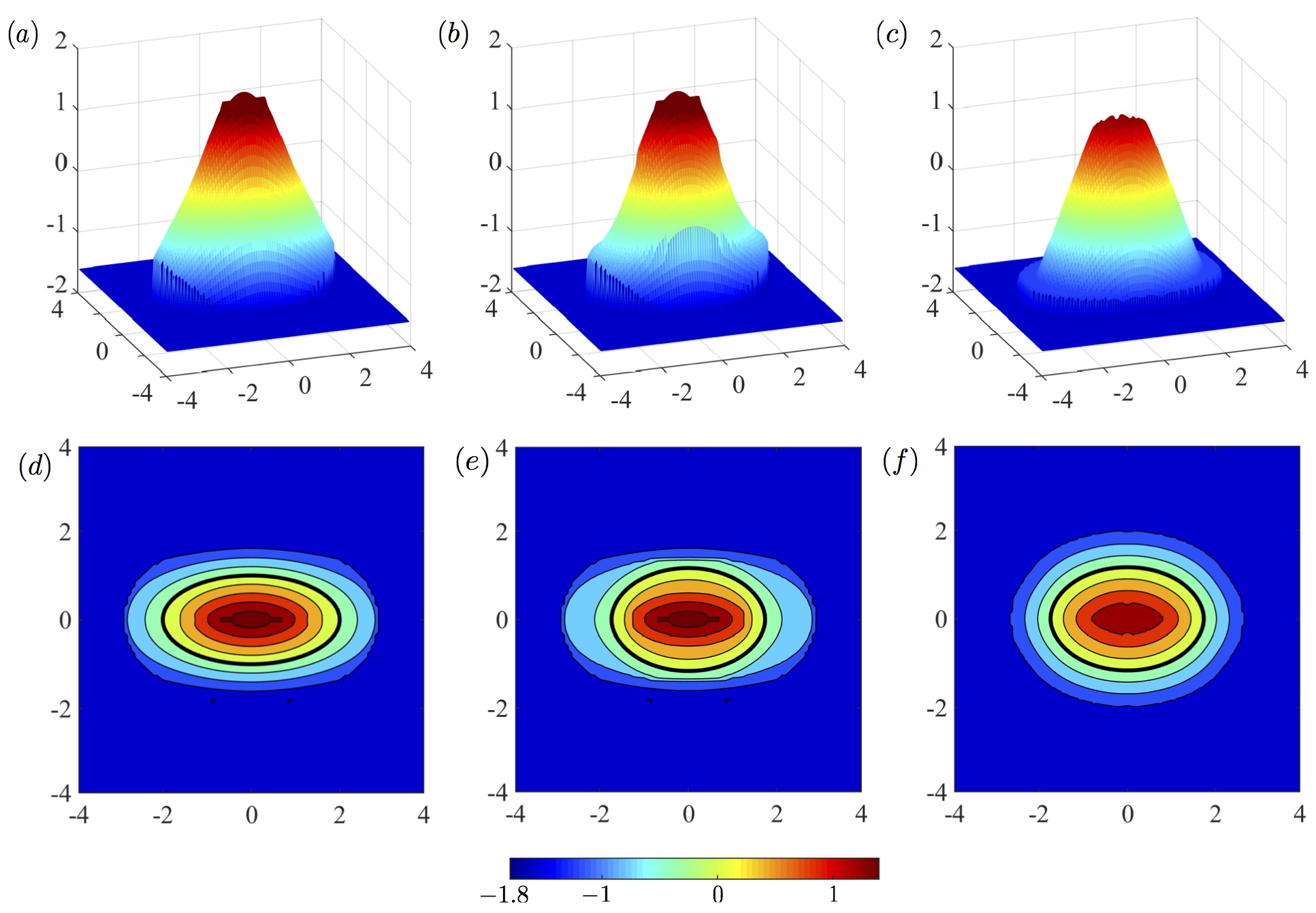}
\caption{Example 2. An initially stretched elastic membrane relaxes in a Navier-Stokes fluid. The fluid-structure interaction is computed in a small region surrounding the membrane (a local level set method). (a,d) The initial level set function, and its contours, with the interface shows as a dark black line. (b,e) The level set function and contours after the membrane has partially relaxed; the membrane encroaches the boundary of the computational tube containing the interface, leading to the development of errors. (c,f) The level set function after reinitialization; the interface is again well separated from the computational boundary, and the stretching information has been preserved.}
\label{fig:phisimu}
\end{center}
\end{figure}
As a second example,  we look at the application of LGPR in the Eulerian immersed boundary method by simulating the dynamics of a relaxing membrane in a fluid, where reinitialization becomes necessary since a local level set method is used: the fluid is described by the Navier-Stokes equations with Reynolds number $Re=1$, which are solved using projection method of Kim and Moin \cite{kimmoin85} on a two-dimensional rectangular grid with no-slip boundary conditions, while the level set function is only constructed and updated in a tube that contains the interface. In its undeformed state the membrane has an arc length parameter $\xi$, where $0\leq \xi< 2\pi$. To begin the simulation the membrane is initially stretched to an ellipse given by $X(\xi)=(2\cos(\xi), \sin(\xi))$. The system evolves under the tension generated by the curve and the stretching energy is given again by linear elasticity, $E(\chi)=(K/2)(\chi-1)^2$ with $K=5$. 

Figs.~\ref{fig:phisimu}(a,d) show the initial level set function and its contours, with the interface shown as a dark black line, on the domain $[-4,4]^2$ with grid size $h=8/100$. As the membrane relaxes over time, the level set function evolves to the one shown in Figs.~\ref{fig:phisimu} (b,e). The level set function is no longer adequate: the interface is close to the boundary of the tube where the local level set method is applied, and the level set function has a large gradient in the computational domain (though not yet at the interface). At this point we perform the LGPR algorithm in another tube that contains the new interface. The resulting level set is shown in Figs.~\ref{fig:phisimu}(c,f). The new level set function is better behaved and the sources of possible numerical error associated with membrane encroachment of the boundary have vanished. The $x$-component of the force before and after reinitialization is shown in Fig.~\ref{fig:fxsimu}. The force is preserved with the expected error. 

\begin{figure}[htbp]
\begin{center}
\includegraphics[width=\textwidth]{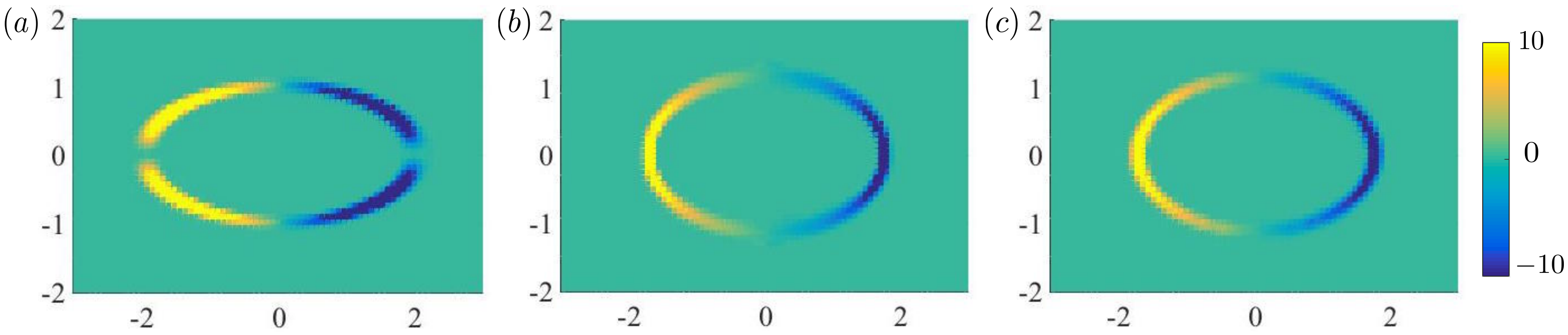}
\caption{The $x$-component of the force from Example 2. (a) The initial force distribution. (b) The force after partial relaxation but before level set reinitialization (c) The force after reinitialization is preserved through the LGPR process.}
\label{fig:fxsimu}
\end{center}
\end{figure}

\subsection{Example 3: Constructing the level set function}
As explained in Section \ref{sec:app}, LGPR is used to initialize the level set function from a given parametric curve and we have done this in example 2. We now focus on the behavior of LGPR in initializing a level set function. We take the interface in the previous example (Recall that we require that the level set function satisfies $|\nabla\phi|(X(\xi))=|X_{\xi}(\xi)|$).  

\begin{figure}[htbp]
\begin{center}
\includegraphics[width=\textwidth]{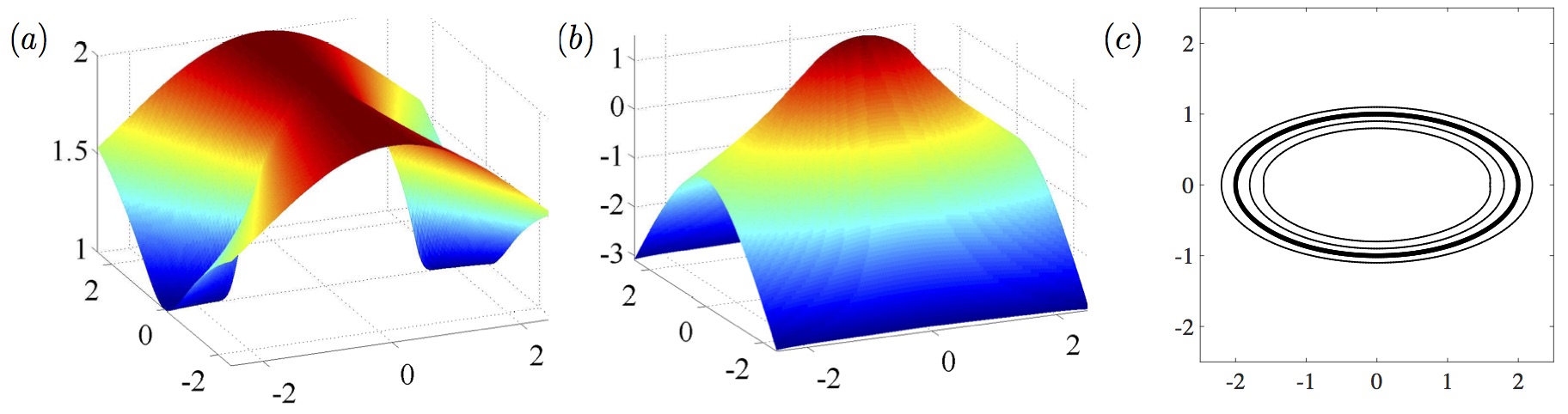}
\caption{(a) The cost function constructed from the initial parametrized curve from Example 2. (b) The constructed level set function. (c) Contours of the level set function; the bold line is the zero level set, or membrane interface.}
\label{fig:fromcurve}
\end{center}
\end{figure}

Fig.~\ref{fig:fromcurve} shows the numerical results on the domain $[-2.5, 2.5]^2$ with spatial grid size $h=5/128$. The characteristic lines intersect on the medial axis of the ellipse and our extension scheme gives good results at the medial axis. The interface is well captured by our constructed level set function (Fig.~\ref{fig:fromcurve}(c)).  To test the accuracy of the location and the interface gradient, we plot in Fig.~\ref{fig:strerr_ex2} the location error $\mathcal{E}_L=\max_p\{|x_p^2/4+y_p^2-1|\}$ and the gradient error $\mathcal{E}_G=\max_p\{||\nabla\phi|(p)-\sqrt{4y_p^2+x_p^2/4}|\}$, where $p$ is a point on the interface, $p=(x_i, y_{\Gamma})$ or $p=(x_{\Gamma}, y_j)$. The interface location is retained with roughly third order accuracy while the interface gradient is retained with second order accuracy, as desired.
\begin{figure}[htbp]
\begin{center}
\includegraphics[width=0.85\textwidth]{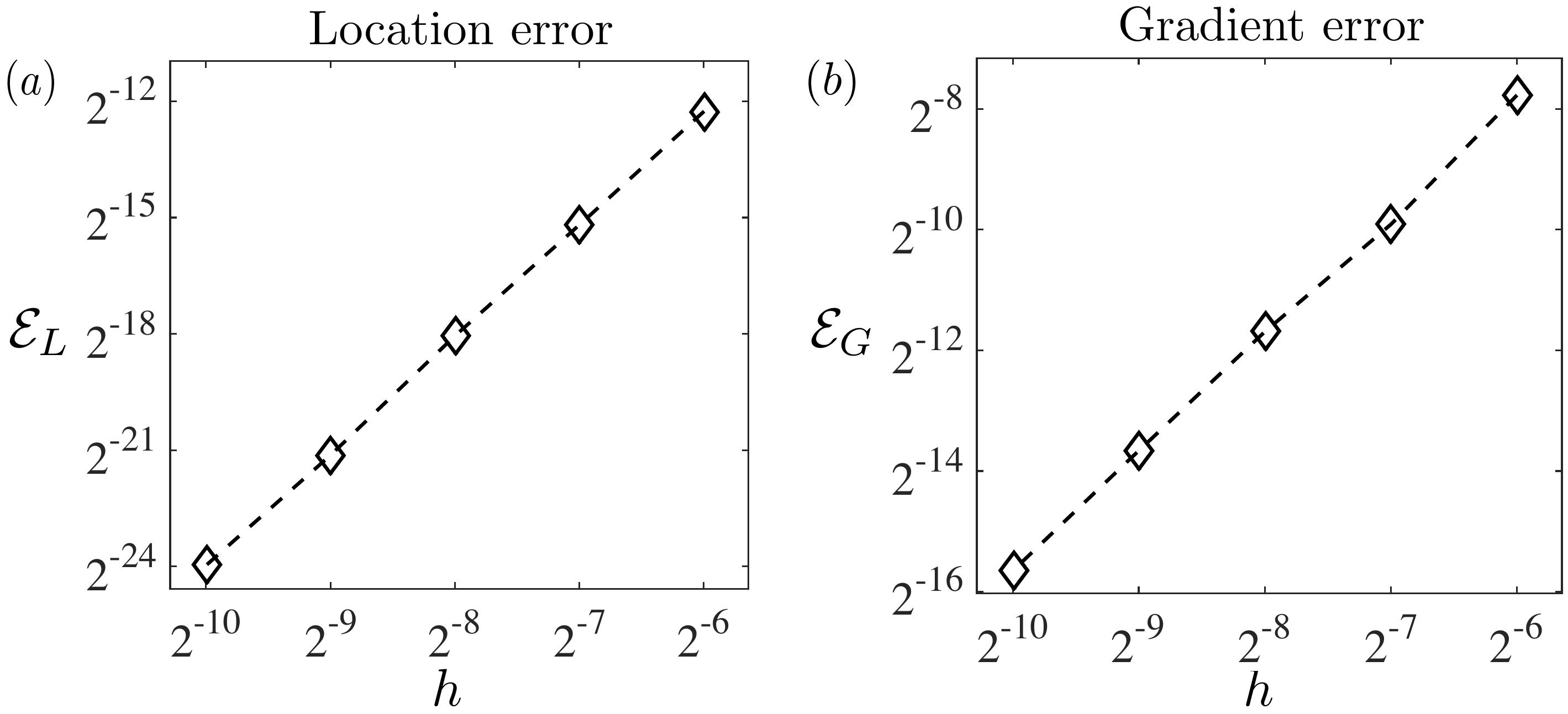}
\caption{(a) The interface location error for the constructed level set function in Example 3 decays as $O(h^3)$. (b) The interface gradient error for the constructed level set function decays as $O(h^2)$.}
\label{fig:strerr_ex2}
\end{center}
\end{figure}

\section{Conclusion}\label{sec:conc}
We have extended the traditional reinitialization process of a level set function to a locally gradient-preserving reinitialization method, which preserves not only the interface location, but also information about tangential interface stretching. We have shown in theory that the proposed method can correctly yield a desired level set function. In particular, we have shown that the viscosity solutions of the reinitialization equation converges to the unique ``proper'' viscosity solution of the Eikonal equation, which is the desired level set function. Numerical schemes are proposed to solve these PDEs. The subcell resolution method for the transport equation is easy to implement and can extend the interface gradient to the whole domain with high accuracy. Numerical examples show that our method is successful, even for discontinuous cost functions. In applications (Eulerian immersed boundary method, for example), a local level set method is desirable for computational cost reductions, and the LGPR method can be applied in a small region containing the interface $\Gamma$. This reinitialization process will make possible the Eulerian formulation based on a local level set method for simulating physical problems where the force depends on stretching in addition to bending. Further details on this particular application will be the topic of a separate paper.


\bibliographystyle{siam}
\bibliography{gradlevelset}
\newpage
\appendix
\section{Cut locus of the interface}

We first show that the signed distance function $\varphi$ is not differentiable at any point in $A$. 
{\bf
\begin{lemma}
Under the assumption we stated in \S\ref{sec:theory}, $\varphi$ is not differentiable at any point in $A$.
\end{lemma}
}
\begin{proof}
Since $\bar{A}\cap\Gamma=\emptyset$, to show that $\varphi$ is not differentiable we show that $d(x,\Gamma)$ is not differentiable at $x\in A\cap U$. (The treatment of $A\cap U^c$ is similar.) By the definition of $A$, we can find $y_1, y_2\in \Gamma$, $y_1\neq y_2$ so that $d(x, y_1)=d(x, y_2)=\varphi(x)$. Suppose that $\varphi$ is differentiable at $x$. We must have $|\nabla\varphi(x)|=1$. Let $\hat{n}=\nabla\varphi(x)$. We have $x+\e \hat{n}\in U$ for small enough $\e>0$. The inequality $d(x+\e\hat{n}, y_1)\geq \varphi(x+\e\hat{n})>\varphi(x)$ then gives
\begin{gather}
d(x+\e\hat{n}, y_1)-d(x, y_1)\geq \e +o(\e),
\end{gather}
and the law of cosines tells us that $\angle(\hat{n}, x-y_1)=0$. But the same argument indicates that $\angle(\hat{n}, x-y_2)=0$, leading to a contradiction. 
\end{proof}
We now show that if $U$ is convex, $d(\bar{A}, \Gamma)=d(A,\Gamma)\ge \inf\{1/\kappa\}>0$. This follows from the following lemma about the local structure of the curve satisfying Assumption \ref{ass:gamma}.
\begin{lemma}\label{lemma:bdd}
For any $a\in A$ and any two points $y,z\in P(a)$, if the portion of $\Gamma$ between $y$ and $z$ is the graph of a function over the segment $yz$, then  $\exists\, w$ on the portion such that $\kappa(w)\mu(a)\ge 1/d(a,\Gamma)$ where $\mu(a)=1$ if $a$ is inside $\Gamma$ and $\mu(a)=-1$ otherwise.
\end{lemma}

\begin{proof}We take the straight line $yz$ to be the $x_1$-axis, and assume at point $z$ that the tangent line of $\Gamma$ has a positive slope with angle $\theta$. The projections of $y$ and $z$ to the $x_1$-axis are denoted by $y_1$ and $z_1$. As the portion of $\Gamma$ between $y$ and $z$ is the graph of a function of $x_1$, then so is $\kappa$: $\kappa=\kappa(x_1)$. By integrating $\kappa$ between $y_1$ and $z_1$ ($\kappa$ is integrable by the assumption on the interface), we have
\begin{gather}
\mu(a)\int_{y_1}^{z_1}\kappa(x_1)dx_1=2\sin(\theta).
\end{gather}
By basic trigonometry, $z_1-y_1=2d(a,\Gamma)\sin(\theta)$. Hence, we can find $\tilde{x}_1\in[y_1, z_1]$ so that
\begin{gather}
\mu(a)\kappa(\tilde{x}_1)\ge \frac{2\sin(\theta)}{2d(a,\Gamma)\sin(\theta)}=\frac{1}{d(a,\Gamma)}.
\end{gather}
\end{proof}

Although we will not use the following ``no cusp'' property in this paper, this is still an interesting property of the geometry of cut locus. This property is especially interesting when we discuss uniqueness of the viscosity solution to the Eikonal equation.\\

\begin{theorem}\label{thm:cusp}
For locally analytical boundary $\Gamma$, there is no cusp with zero angle in $\bar{A}$.
\end{theorem}

\begin{proof}
We will sketch the proof without details. Let us \deleted{We} take $\bar{A}\cap U$ as the example. Suppose that there are two edges making a zero angle at $x\in \bar{A}\cap U$. Parametrize them with arc lengths: $\gamma_1(s)$ and $\gamma_2(s)$, such that $\gamma_1(0)=\gamma_2(0)=x$,$\gamma_1'(0)=\gamma_2'(0)=:\hat{n}$ and that $\forall \e>0$, $\exists s_{\e}\in(0,\e)$, $\gamma_1(s_{\e})\neq \gamma_2(s_{\e})$. By Lemma \ref{lemma:proj}, there is an opening region between $\gamma_1$ and $\gamma_2$. Choose a sequence $x_k$ approaching $x$ inside the region such that the projection of $Px_k$ (which is one point) approaches $y$. Then $y\in P(x)$. By Lemma \ref{lemma:proj}, $x_k$-$Px_k$ doesn't intersect with $\bar{A}$. As a limit line segment, $x$-$y$ doesn't cross $\gamma_1,\gamma_2$ and thus is tangent to the two curves, namely $\hat{n}=(y-x)/|y-x|$. $\gamma_1$ and $\gamma_2$ must be on the two sides of the line $xy$. Consider $\gamma_1$ and the portion of $\Gamma$ on that side, parametrized as $\Gamma_0(u)$ with $\Gamma_0(0)=y$. Then there exists a $\delta>0$ such that the curvature is a smooth function on $[0,\delta]$ (if necessary, redefine $\kappa(0)=\lim_{u\to 0^+}\kappa(u)$). Choose $s_k\to 0^+$ such that $\#P\gamma_1(s_k)\ge2$. From the fact that $x-y$ is the tangent line of $\gamma_1$ at $x$, $x-\gamma_1(s_k)$ and $\gamma_1(s_k)-y$ are almost parallel for sufficiently large $k$. Then, for $k$ large enough, $w_k\in P\gamma_1(s_k)$, $d(x,w_k)\ge d(x,y)$ implies that $w_k$ must fall onto $\Gamma_0(s)$ for $0\le s\le \delta$. Using $\gamma_1(s_k)=x+s_k\hat{n}+O(s_k^2)$ and $P(x+s_k\hat{n})=y$, we know $\forall w_k\in P\gamma_1(s_k)$ that $|w_k-y|=O(s_k^2)$.  By Lemma \ref{lemma:bdd} and since $\delta$ is small, there is a $u_k$ such that $\Gamma_0(u_k)$ is between two points in $P\gamma_1(s_k)$ and $\kappa(u_k)\ge 1/d(\gamma_1(s_k), \Gamma)$. It is clear that $u_k=O(s_k^2)$. The contradiction is obtained by noticing that $\kappa(0)\le 1/d(x,y)$ and $O(s_k^2)=|\kappa(u_k)-\kappa(0)|\ge |1/d(\gamma_1(s_k), \Gamma)-1/d(x,y)|\ge Cs_k$ for some $C>0$.
\end{proof}

\section{Solutions of Eikonal equation with discontinuous cost function}

In this section, we will prove the existence and uniqueness of the viscosity solution to \eqref{eq:eikonal} (a.k.a., Theorem \ref{thm:37}). 

The existence part of the proof is quite standard. However, the proof of the uniqueness is different from the standard argument. One may find that there is one essential difference in our problem: in our case, the cut locus may have bifurcation points. Which means, there could exist point $x$ in the singular set of cost function, such that for any disc $B$ centering at $x$, $B$ will be divided by the singular set into more than two parts. This makes the arguments in \cite{ostrov00,soravia02, soravia06, de04} invalid.\\
On the other hand, there is another subtle difference in our proof. Although from Theorem \ref{thm:cusp}, there is no cusp in the singular set of the cost function, we will not use this. In other words, our proof doesn't need the classical ``no cusp'' assumption for the singular set of cost function.

\subsection{Existence}
We take $x\in U$ as an example case. We show that $\phi_M$ is a viscosity solution; a similar argument applies to $\phi_m$. We have the following relationships:
\begin{multline}
c_1\leq\essinf\{f(y)| y\in B(x,\sqrt{(c_2-c_1)\e})\}
\leq f_{\e}(x)\leq f_*(x)\\
\leq f^*(x)\leq f^{\e}(x)\leq\esssup\{f(y)| y\in B(x,\sqrt{(c_2-c_1)\e})\}\leq c_2.
\end{multline}
$f_{\e}$ increases to $f_*$ while $f^{\e}$ decreases to $f^*$, and $f^{\e}$($f_{\e}$) is continuous. The corresponding solution $\phi^{\e}$ is given as in \eqref{eq:valuefunc}. One then could verify the following dynamic programming principle: for $x\in U$ and $0\leq s_1\leq d(x, \Gamma)$,
\begin{gather}
\phi^{\e}(x)=\inf_{\gamma\in \mathscr{C}}\left\{\int_{0}^{s_1}f^{\e}(\gamma(s))ds+\phi^{\e}(\gamma(s_1))|\gamma(0)=x\right\}.
\end{gather}
The proof of this principle is omitted here. For similar arguments, one could refer to Chapter 10 of \cite{evans10}.  A direct conclusion from this principle is that
\begin{gather}
\begin{array}{c}
0\leq \phi^{\e}\leq c_2 d(x,\Gamma)\\
|\phi^{\e}(x)-\phi^{\e}(y)|\leq c_2|x-y|
\end{array}
\end{gather}
By the Arzela-Ascoli theorem, since $\bar{U}$ is compact, there is a subsequence $\phi^{\e_k}$ that converges uniformly to $\bar{\phi}$.

Now suppose that $\bar{\phi}-\zeta$ has a local maximum at $x_0\in U$. For small $\delta>0$, $\bar{\phi}-\zeta-\delta|x-x_0|^2$ has a strict local maximum at $x_0$. There is a sequence of local maxima $x_k$ for $\phi^{\e_k}-\zeta-\delta |x-x_0|^2$ that converges to $x_0$ by the uniform convergence of the function sequence. Fixing $K>0$, for $k>K$, we have 
\begin{gather}
|\nabla\zeta(x_k)+2\delta(x_k-x_0)|\leq f^{\e_k}(x_k)\leq f^{\e_K}(x_k).
\end{gather}
Letting $k\to\infty$, since $f^{\e_K}$ is continuous, $|\nabla\zeta(x_0)|\leq f^{\e_K}(x_0)$. Sending $K\to\infty$, we obtain
$|\nabla\zeta(x_0)|\leq f^*(x_0)$. For a local minimum at $x_0$, the argument is similar. Here, we just use the modified function $\bar{\phi}-\zeta+\delta |x-x_0|^2$ and the inequalities
\begin{gather}
|\nabla\zeta(x_k)-2\delta(x_k-x_0)|\geq f^{\e_k}(x_k)\geq f_{\e_k}(x_k)\geq f_{\e_K}(x_k)
\end{gather}
for $k>K$. Thus, $|\nabla\zeta(x_0)|\geq f_*(x_0)$ and $\bar{\phi}$ is a viscosity solution.

Now, fixing any $\delta>0$, we can find $K>0$ so that $\bar{\phi}(x)+\delta>\phi^{\e_k}(x)$ when $k>K$ for any $x\in U$. However,  $\exists\,\gamma$ with $\gamma(0)=x$ such that $\phi^{\e_k}+\delta>\int_0^Lf^{\e}(\gamma(s))ds\geq \int_0^Lf^*(\gamma(s))ds$. ($f^*(\gamma(\cdot))$ is the infimum of continuous functions and thus Lebesgue measurable on $[0, T]$.) Meanwhile, $\bar{\phi}(x)-\delta\leq \phi^{\e_k}(x)\leq \int_0^L f^{\e_k}(\gamma(s))ds$ for $k$ large enough and any $\gamma\in\mathscr{C}$ with $\gamma(0)=x$ and $\gamma(L)\in\Gamma$. Now fixing $\gamma$ and taking $k\to\infty$, the dominant convergence theorem tells us that
$\bar{\phi}(x)-\delta\leq \int_0^Lf^*(\gamma(s))ds$. Hence:
\begin{gather}
\bar{\phi}(x)=\inf_{\gamma\in\mathscr{C}}\left\{\int_0^{L}f^*(\gamma(s))ds |\gamma(0)=x, \gamma(L)\in\Gamma \right\}=\phi_M(x).
\end{gather}

The dynamic programming principle for $\phi_M$ is still true and $\phi_M$ is also bounded and Lipschitz continuous with the same constants.
\subsection{Uniqueness}
We again fix $x\in U$, and we now show that $\phi_m(x)=\phi_M(x)$.  For any $\e>0$, we can find $\gamma\in \mathscr{C}$ with $\gamma(0)=x$ and $\gamma(L)\in\Gamma$ so that $\int_0^Lf_*(\gamma(s))ds<\phi_m(x)+\e$. $\gamma$ has no self-intersection by the definition of $\mathscr{C}$.

By Lemma \ref{lmm:A}, except at finitely many points, all the points in $\bar{A}$ belong to some locally analytical curves and have a projection with size $2$. Noticing that $\gamma$ is an injection, we pick a set $E$ that covers the finite irregular points such that the total length of $\gamma$ falling into $E$ is less than $\e/c_2$, where $c_2$ is the upper bound of $f^*$. Then the remaining part $\bar{A}\setminus E$ has the following properties: it is the disjoint union of $N$ edge portions $e_n, 1\le n\le N$; for any $x\in e_n$, we can find a ball $B(x, r_x)$ ($r_x>0$) so that every point in $\bar{A}\cap B(x,r_x)$ has a projection of size $2$ and $\bar{A}\cap B(x,r_x)$ is real analytic. \\
\\
{\it Step 1. We first show that the cost function $f$ has limits on both sides of the edge $e_n$.} \\

For any $x\in e_n$, $B(x,r_x)$ is divided into two subdomains $B_1$ and $B_2$ by $\bar{A}$. Let $x_k\in B_1\cap U$, $x_k\to x$. The sequence $w_k=Px_k$ has a limit point $z\in Px$. Further inspection reveals that $z$ is the only limit point of $w_k$ since $\#Px=2$ and the sequences in $B_2$ give another. This means that for every sequence in $B_1$ converging to $x$, the projections converge to $z$. Hence, $\lim_{y\to x, y\in B_1} f(y)=\chi(z)$. If the limit function is $f_1$, by the continuity of projection on one side, $f_1$ is continuous on $e_n$. $f_2$ may be similarly defined.\\
\\
{\it Step 2. We now decompose $\bar{A}$ into several parts so that on each part $\int(f^*-f_*)(\gamma(s))ds$ can be dealt with appropriately.}\\

  Let $e_n$ be equipped with the $1$D Lebesgue measure $m$ induced by the arc length, and let $F_n=\{x\in e_n: f_1(x)=f_2(x)\}$. Clearly, $f^*=f_*$ on $F_n$ and $F_n$ is closed. The set $e_n\setminus F_n=\{x\in e_n: f_1-f_2>0, \mbox{or } f_1-f_2<0\}$ is open, thus is the disjoint union of countable subintervals in $e_n$. Since the sum of lengths of these subintervals is finite, we can find finitely many of them, say $M_n$ of them, such that the measure of the remaining is small. For these $M_n$ intervals, we can cover the endpoints and get $M_n$ new subintervals denoted as $I_i, 1\le i\le M_n$. Hence, we can decompose $e_n$ into $G_n$ with $m(G_n)<\e/Nc_2$, $F_n$ and $\cup_{i=1}^{M_n}I_i$. Let's consider all the $M=\sum_n M_n$ subintervals. $\exists\delta_1,\delta>0$ that depend on $E$, $F_n, G_n, 1\le n\le N$, such that each subinterval $I_i$ satisfies: $U_i=\{x: d(x, I_i)\le\delta_1\}$ is divided into two domains $V_{i1}, V_{i2}$, $f^*\in C(\bar{V}_{i2})$, $f_*\in C(\bar{V}_{i1})$, and $\inf_{x\in V_{i2}, y\in V_{i1}} f^*(x)-f_*(y)\ge\delta$. We have thus decomposed $\bar{A}$ into the union of following sets: $M=\sum_n M_n$ subintervals; a closed set on which $f_*=f^*$; and a set (union of $G_n$ and $E$) with measure less than $2\e/c_2$. 

Let $C_i=\{s: \gamma(s)\in I_i\}$ be a subset of $[0,L]$ and $C=\cup_{i=1}^MC_i$. It is clear that $\int_{[0,L]\setminus C}(f^*-f_*)(\gamma(s))ds<2\e$. Hence, we only need to study the integral on $C_i$. \\
\\
{\it Step 3. We now obtain a local property of the portion of $\gamma$ on $C_i$.}\\

 By the local smoothness of the edges, $\exists\alpha>0$, $\delta_3<\delta_1/2$, such that if $s_1, s_2\in C_i$, with $s_2-s_1\le \delta_3$, then the length of $I_i$ between $\gamma(s_1)$ and $\gamma(s_2)$ is at most $(s_2-s_1)+\alpha(s_2-s_1)^2$.  If $\exists s\in(s_1,s_2)$ so that $\gamma(s)\notin \bar{V}_{i1}$, then there's an subinterval $[s_3,s_4]\subset [s_1,s_2]$ such that $s\in [s_3,s_4]$ and $\gamma((s_3,s_4))\cap \bar{V}_{i1}=\emptyset$. Since $\delta_3<\delta_1/2$, $\gamma([s_3,s_4])$ can't leave $\bar{V}_{i2}$. Noticing $f^*=f=f_*$ on $\gamma((s_3,s_4))$, 
$$\int_{\gamma[s_3,s_4]}f_*ds=\int_{\gamma[s_3,s_4]}f^*ds>\int_{J}f_*ds+(s_4-s_3)\delta-\alpha(s_4-s_3)^2c_2$$ where $J$ is the part of $I_i$ between $\gamma(s_3)$ and $\gamma(s_4)$. If we pick $\delta_3$ small enough, we could have $(s_4-s_3)\delta-\alpha(s_4-s_3)^2c_2>0$. We replace $\gamma([s_3,s_4])$ with $J$ and get a new curve $\tilde{\gamma}$, and we see that $\int_{\tilde{\gamma}}f_*ds<\int_{\gamma}f_*ds$. $\tilde{\gamma}$ may be self-intersecting.  We can modify it as following: If $J$ intersects with $\gamma(0,s_3)$. we find the infimum of $s_{\min}$ on $[0,s_3]$ so that $\gamma(s_{\min})\in J$. Then we piece $\gamma[0, s_{\min}]$ and the part of $J$ starting from $\gamma(s_{\min})$ together. Then, by the same method we can deal with the case when $J$ intersects with $\gamma((s_4,L))$. Such $s_3,s_4$ pairs correspond to disjoint open intervals, so we can do this modification at most countable many times and get another curve $\gamma_1\in\mathscr{C}$, so that $\int_{\gamma_1}f_*ds<\int_{\gamma}f_*ds$. Hence, without loss of generality we can assume $\gamma$ satisfies this property: 

If $s_1, s_2\in C_i, |s_2-s_1|<\delta_3$, then $\gamma([s_1, s_2])\subset \bar{V}_{i1}$.\\
\\
{\it Step 4. With the property obtained, we perturb the curve defined on $C_i$ so that the integral on $C_i$ can be treated.}\\

$C_i$ is closed and consists of countable closed intervals (they are subintervals of $[0,L]$, different from the intervals on the edge portion) and a nowhere dense set $G_i'$ ($G_i'$ may have positive measure). Moreover $G_i=\bar{G}_i'$ is still nowhere dense since the extra possible points are the endpoints of the intervals. By the assumption above, we may write $G_i=\cup_{j=1}^{N_1} G_{ij}$ for some $N_1\in\mathbb{N}$. $[\inf G_{ij}, \sup G_{ij}]\setminus G_{ij}$ does not contain any interval of length $\ge\delta_3$ for any $j$, so that $\gamma([\inf G_{ij}, \sup G_{ij}]) \subset \bar{V}_{i1}$, and
$$
\inf_{v_1\in G_{ij_1}, v_2\in G_{ij_2}}|v_1-v_2|\ge\delta_3
$$
Now, assume $K$ is $G_{ij}$ or one of the intervals in $C_i$, and let $s_l=\inf K, s_r=\sup K$ and $\e_1>0$ be fixed. We can find finitely many points $s_{k}$ in $K$ and the difference between two consecutive points is less than $\delta_3$. We shift $\gamma([s_{k}, s_{k+1}])$ along the normal of $I_i$ at $\gamma(s_{k})$ toward $V_{i1}$ with distance $\delta_4$. Now, we add line segments to connect the endpoints. The shifted curve portions and line segments are all in $V_{i1}$ if $\delta_4$ is small enough. Denote the curve so obtained by $\gamma'$. By the uniform continuity of $f_*$ in $\bar{V}_{i1}$, we have $\int_{\gamma'}f_*ds<\int_{\gamma}f_*ds+\e_1$.  However, $\gamma'$ may be self-intersecting. We now modify it following an essentially similar procedure as before: $\gamma'$ consists of $\gamma[0,s_l]$, $\gamma[s_r, L]$ and the shifted curves together with line segments, denoted as $\tilde{P}$. Consider the first shifted curve portion with the line segment $\tilde{P}_{1}$. Suppose $k\ge 2$ is the largest number such that $\tilde{P}_{k}$ intersects with $\tilde{P}_{1}$. We find the first point on $\tilde{P}_1$ that is on $\tilde{P}_k$, then discard the part on $\tilde{P}_1$ after this point and all curve portions $\tilde{P}_l$ with $1<l<k$ and the part on $\tilde{P}_k$ that is before this point. Since the portions are finite, this process can be terminated, resulting in $P\subset\tilde{P}, P\in\mathscr{C}$ that connects $\gamma(s_l)$ and $\gamma(s_r)$. The remaining work is the same as how $J$ was modified before. Then, we find a new curve $\gamma_2$. $\int_{\gamma_2}f_*ds\le \int_{\gamma'}f_*ds$. By the construction, we have removed $K\setminus\{s_l,s_r\}$ from $C_i$ without adding new points. Such sets are countable, we can finish this process and obtain a curve $\gamma_3$, so that $C_i(\gamma_3)$ consists of countably many points and $\int_{\gamma_3}f_*ds-\int_{\gamma}f_*ds<\e/M$ since $\e_1$ is arbitrary. \\

Hence, we are able to construct a curve $\gamma_4$ with $\gamma_4(0)=x$ and $\gamma_4(L)\in\Gamma$ (since it has the same endpoints as $\gamma$) such that $\int_{\gamma_4}f^*ds\le\int_{\gamma_4}f_*ds+2\e<\int_{\gamma}f_*ds+3\e<\phi_m(x)+4\e$, which verifies the condition so that $\phi_m(x)=\phi_M(x)$.

\section{The solution of the reinitialization equation}

In this section we show
that the formula given in \eqref{eq:solhj} is a viscosity sub-solution of the level-set reinitialization equation. Similar argument shows that it is also a super-solution is similar.

Consider that $u-\zeta$, where $\zeta$ is $C^{\infty}$, has a local maximum at $(x_0, \tau_0)$ for $\tau_0>0$. We show that the sub-solution condition is satisfied. If $\tau_0>\tau_{x_0}$, then in a neighborhood of $(x_0,\tau_0)$, we have $\tau>\tau_x$ since as $\tau_x$ is continuous on $x$. The solution does not depend on $\tau$, and we also must have $\zeta_{\tau}(x_0,\tau_0)=0$. The sub-solution condition is satisfied for $x_0\notin\Gamma$ as the condition has already been verified for the Eikonal equation. If $x_0\in\Gamma$, $[\sgn(u_0)(|p|-g)]_*|_{p=\nabla\zeta(x_0), x=x_0}\leq 0$ is assured. 

Consider $0<\tau_0\leq \tau_{x_0}$. Then $u_0(x_0)\neq 0$ since for any $x\in\Gamma$, we have $\tau_x=0<\tau_0$. Take $u_0(x_0)>0$ (the result for $u_0(x_0)<0$ is similar). Let $h_1=\min\{\tau_0, d(x_0,\Gamma)\}>0$. There exists $0<h_2\leq h_1$, such that the dynamical programming principle holds for $h<h_2$:
\begin{gather}
u(x_0, \tau_0)=\inf_{\gamma}\left\{u(\gamma(h), \tau_0-h)+\int_0^hg(\gamma(s))ds |\gamma(0)=x_0 \right\}.
\end{gather} 
We now show this principle. Since $\tau_0\leq \tau_{x_0}$, for any $\e>0$, we can find $\gamma$ with $\gamma(0)=x_0$ such that
$u(x_0,\tau_0)+\e>u_0(\gamma(\tau_0))+\int_0^{\tau_0}g(\gamma(s))ds$. By the definition, $\int_h^{\tau_0}g(\gamma(s))ds+u_0(\gamma(\tau_0))\geq u(\gamma(h), \tau_0-h)$ whether or not $\tau_0-h>\tau_{\gamma(h)}$. `$\geq$' is thus shown.

We now show the other direction. Let $B=\inf_{\gamma}\{\int_0^Lg(\gamma(s))ds |\gamma(0)=x_0,\gamma(L)\in\Gamma\}$, and fix an arbitrary $\gamma\in\mathscr{C},\gamma(0)=x_0$. We will discuss both cases where either $\tau_0<\tau_{x_0}$ or not.

Consider $\tau_0<\tau_{x_0}$. Take $h_2\leq h_1$ small enough so that $|x-x_0|<h_2, |\tau-\tau_0|<h_2$ implies $\tau<\tau_x$ due to the continuity of $\tau_x$.  Let $h<h_2$. We can find $\gamma_1,\gamma_1(0)=\gamma(h)$ such that $u(\gamma(h), \tau_0-h)+\e>\int_0^{\tau_0-h}g(\gamma_1(s))ds+|u_0|(\gamma_1(\tau_0-h))$ since $\tau_{\gamma(h)}>\tau_0-h$. Connecting $\gamma(s): 0\le s\le h$ and $\gamma_1$, we find $\gamma_2$. $\gamma_2(\tau_0)=\gamma_1(\tau_0-h)$. Then, $u(x_0,\tau_0)\leq \int_0^{\tau_0}g(\gamma_2(s))ds+|u_0(\gamma_2(\tau_0))|<\int_0^hg(\gamma(h))+u(\gamma(h), \tau_0-h)+\e$.

We now assume that $\tau_0=\tau_{x_0}$ (recall that we are discussing the case where $\tau_0\le \tau_{x_0}$). We must have that $u(x_0,\tau_0)=B$. Let $h<h_1$. If $\tau_0-h\geq\tau_{\gamma(h)}$, then $\exists\gamma_1,\gamma_1(0)=\gamma(h)$, $\gamma_1(L)\in \Gamma$ where $L\leq \tau_0-h$ such that $u(\gamma(h),\tau_0-h)+\e>\int_0^Lg(\gamma_1(s))ds$. Then, connecting $\gamma(s): 0\le s\le h$ and $\gamma_1$, we get a new curve $\gamma_3$ with total length $h+L\leq\tau_0$. We then have $u(x_0,\tau_0)=B\leq \int_0^{L+h}g(\gamma_3(s))ds<\int_0^hg(\gamma(s))ds+u(\gamma(h),\tau_0-h)+\e$. If $\tau_0-h< \tau_{\gamma(h)}$, we can find $\gamma_1,\gamma_1(0)=\gamma(h)$ such that $u(\gamma(h), \tau_0-h)+\e>\int_0^{\tau_0-h}g(\gamma_1(s))ds+|u_0|(\gamma_1(\tau_0-h))$. Connecting $\gamma(s): 0\le s\le h$ and $\gamma_1$, we get $\gamma_2$. Since $\tau_0=\tau_{x_0}$, $B\le \int_0^{\tau_0}g(\gamma_2(s))ds+|u_0(\gamma_2(\tau_0))|$ by the definition of $u(x_0,\tau_0)$. The same argument as in the previous paragraph follows.

Combining what we have, the dynamic programming principle is true. With this principle the sub-solution condition is easily verified: since $u(x_0,\tau_0)-\zeta(x_0,\tau_0)\geq u(\gamma(h), \tau_0-h)
-\zeta(\gamma(h), \tau_0-h)$, we see that
\begin{gather}
\zeta(x_0,\tau_0)\leq \inf_{\gamma}\left\{\zeta(\gamma(h), \tau_0-h)+\int_0^hg(\gamma(s))ds |\gamma(0)=x_0 \right\},
\end{gather}
which implies that
\begin{gather}
\zeta_{\tau}(x_0,\tau_0)+|\nabla\zeta|(x_0,\tau_0)\leq g(x_0).
\end{gather}
\end{document}